\RequirePackage{fix-cm}
\documentclass[smallextended]{svjour3opt} 
\usepackage{todonotes}
\usepackage{makecell}
\usepackage{palatino,url}
\usepackage{verbatim, amssymb,amsmath}
\usepackage{varwidth}
\usepackage[export]{adjustbox}
\usepackage{multirow}
\usepackage{mathrsfs}
\usepackage{graphics,amsopn,amstext,amsfonts,color}
\usepackage{bm}
\usepackage{setspace}
\usepackage{graphicx}
\usepackage{epstopdf}
\usepackage{multirow}
\usepackage{booktabs}
\usepackage{paralist}
\usepackage{colortbl}
\usepackage{appendix}
\usepackage{soul}
\usepackage[caption=false]{subfig}
\usepackage[hidelinks,bookmarksnumbered=true]{hyperref}
\usepackage{graphics,latexsym,amsfonts,verbatim,amsmath}
\usepackage{algorithm}
\usepackage{algorithmicx}
\usepackage{algpseudocode}
\usepackage{tikz}
\usetikzlibrary{calc}
\usetikzlibrary{plotmarks}
\usetikzlibrary{backgrounds}

\usepackage{pgfplots}
\usepackage{bbm}

\usepackage{pgf}
\usepackage{pst-func}

\definecolor{pinegreen}{rgb}{0.0, 0.47, 0.44}
\usepackage{pgfpages}
\usepackage{hhline} 
\usepackage{empheq}
\usepackage{cases}

\usepackage{multirow,todonotes}

\usepackage{epstopdf}

\usepackage{xcolor}

\usepackage{todonotes}
\usepackage{bm}
\usepackage{tikz}
\usepackage{pgfplots}
\usepgfplotslibrary{fillbetween}
\pgfplotsset{compat=newest}
\pgfdeclareplotmark{mystar}{
\node[star,star point ratio=2.25,minimum size=6pt,
inner sep=0pt,draw=black,solid,fill=red] {};
}
\usetikzlibrary{decorations.pathreplacing}

\usepgflibrary{arrows}

\usetikzlibrary{calc}
\usetikzlibrary{plotmarks}
\usepackage{wrapfig}

\usetikzlibrary{shapes,decorations,arrows,calc,arrows.meta,fit,positioning,patterns}
\tikzset{
-Latex,auto,node distance =1 cm and 1 cm,semithick,
state/.style ={ellipse, draw, minimum width = 0.7 cm},
point/.style = {circle, draw, inner sep=0.04cm,fill,node contents={}},
bidirected/.style={Latex-Latex,dashed},
el/.style = {inner sep=2pt, align=left, sloped}
}

\def\E{{\mathbb E}}

\def\Pr{{\mathbb{P}}}

\def\Re{\mathbb{R}}
\def\Qe{\mathbb{Q}}

\def\hat{\widehat}

\def \P{\mathcal{P}}
\def \F{\mathcal{F}}

\def \Ze{{\mathbb{Z}}}

\def\M{{\mathcal M}}

\def\F{{\mathcal F}}

\def\P{{\mathcal P}}

\def\X{{\mathcal X}}
\def\Y{{\mathcal Y}}

\def\U{{\mathcal U}}
\def\Re{{\mathbb R}}

\newcommand{\e}{\mathbf{e}}

\newtheorem{assumption}{Assumption}

\def\epi{{\mathrm{epi}}}

\newcommand{\trxi}{\tilde{\bm{\xi}}}
\newcommand{\trzeta}{\tilde{\bm{\zeta}}}

\newcommand{\rzeta}{{\bm{\zeta}}}

\DeclarePairedDelimiter\ceil{\lceil}{\rceil}
\DeclarePairedDelimiter\floor{\lfloor}{\rfloor}
\DeclareMathOperator{\sign}{sign}

\DeclareMathOperator{\conv}{conv}

\renewcommand{\S}{\mathcal{\bm S}}

\newcommand{\rxi}{\bm{\xi}}
\usepackage{todonotes,bm}

\usepackage{comment}

\usepackage{enumitem}
\usepackage{bigints}

\def\x{\vect{x}}
\usepackage{thmtools}
\usepackage{thm-restate}

\renewcommand{\S}{\mathcal{S}}

\newcommand{\exclude}[1]{}
\newcommand*{\QEDA}{\hfill\ensuremath{\square}}

\allowdisplaybreaks

\definecolor{mygreen}{RGB}{28,172,0} 
\definecolor{mypurple}{RGB}{170,55,241}
\definecolor{myorange}{rgb}{0.77,0.38,0.06}

\usepackage[numbers,sort&compress,comma]{natbib}

\allowdisplaybreaks

\title{On Tractability, Complexity, and Mixed-Integer Convex Programming Representability of Distributionally Favorable Optimization}
\date{\today}
\titlerunning{MICP-R of Distributionally Favorable Optimization}

\author{Nan Jiang \and Weijun Xie}

\institute{First Author: Nan Jiang \at
Affiliation: Georgia Institute of Technology, Atlanta, GA\\
\email{nanjiang@gatech.edu}
\and Corresponding Author: Weijun Xie \at
Affiliation: Georgia Institute of Technology, Atlanta, GA\\
\email{wxie@gatech.edu}
}

\allowdisplaybreaks

\makeatletter
\pgfmathdeclarefunction{erf}{1}{%
\begingroup
\pgfmathparse{#1 > 0 ? 1 : -1}%
\edef\sign{\pgfmathresult}%
\pgfmathparse{abs(#1)}%
\edef\x{\pgfmathresult}%
\pgfmathparse{1/(1+0.3275911*\x)}%
\edef\t{\pgfmathresult}%
\pgfmathparse{%
1 - (((((1.061405429*\t -1.453152027)*\t) + 1.421413741)*\t
-0.284496736)*\t + 0.254829592)*\t*exp(-(\x*\x))}%
\edef\y{\pgfmathresult}%
\pgfmathparse{(\sign)*\y}%
\pgfmath@smuggleone\pgfmathresult%
\endgroup
}
\makeatother

\begin{document}
\maketitle

\begin{abstract}
Distributionally Favorable Optimization (DFO) is an important framework for decision-making under uncertainty, with applications across fields such as reinforcement learning, online learning, robust statistics, chance-constrained programming, and two-stage stochastic optimization without relatively complete recourse. In contrast to the traditional Distributionally Robust Optimization (DRO) paradigm, DFO presents a unique challenge-- the application of the inner infimum operator often fails to retain the convexity. In light of this challenge, we study the tractability and complexity of DFO. We establish sufficient and necessary conditions for determining when DFO problems are tractable or intractable. Despite the typical nonconvex nature of DFO problems, our findings show that they are mixed-integer convex programming representable (MICP-R), thereby enabling solutions via standard optimization solvers. Finally, we numerically validate the efficacy of our MICP-R formulations. 
\end{abstract}

\vspace{0.2in}

\section{Introduction}
We study Distributionally Favorable Optimization (DFO) 
that admits the following form:
\begin{equation}
v^*=\min_{\bm x \in \X}\inf_{\Pr\in\P}\E_{\Pr}\left[Q(\bm x, \trxi)\right], \label{dfo}
\end{equation}
where 
set $\X\subseteq \Re^n$ is deterministic, set $\P$ denotes an ambiguity set with $\P \subseteq \{ \Pr\colon\Pr\{\trxi\in {\U}\}=1 \}$ and support $\U\subseteq\Re^m$ (also called ``uncertainty set" throughout this paper), and function $Q(\bm x, {\trxi}):\Re^n\times \U\to \Re$. 
Note that if only support information $\U$ is available (i.e., $\P = \{ \Pr\colon\Pr\{\trxi\in {\U}\}=1 \}$), then DFO \eqref{dfo} simplifies to a singular DFO (sDFO), i.e.,
\begin{align}
v^*=\min_{\bm x\in\X}  \inf_{\rxi\in\U} Q(\bm x,\rxi).\label{eq_uq_simple}
\end{align}
The special cases of DFO \eqref{dfo} and their variants have been applied to various fields, including reinforcement learning \cite{agarwal2020optimistic,auer2002finite,song2020optimistic}, image reconstruction \cite{hanasusanto2017ambiguous}, Bayesian optimization \cite{nguyen2019calculating,nguyen2019optimistic,nguyen2020robust}, among others.
Recent advancements in \cite{jiang2023dfo} demonstrate that DFO \eqref{dfo} can recover many robust statistics and machine learning problems. 

It is worth noting that as an opposite counterpart of DFO \eqref{dfo}, the Distributionally Robust Optimization (DRO) of the form
\begin{equation}
\min_{\bm x \in \X}\sup_{\Pr\in\P}\E_{\Pr}\left[Q(\bm x, \trxi)\right], \label{dro}
\end{equation}
has recently successfully addressed many decision-making problems under uncertainty to achieve better out-of-sample performance guarantees (see more discussions in \cite{esfahani2018data,sun2023distributionally,shafieezadeh2015distributionally,mohajerin2018data}). 
A natural way to bridge DFO \eqref{dfo} and DRO \eqref{dro} is through the so-called Hurwicz criterion, proposed in \cite{hurwicz1951generalized,arrow1972optimality}, which can be used to effectively balance the tradeoff between DFO \eqref{dfo} and DRO \eqref{dro}. 
The Hurwicz model has been studied in many decision-making problems (see, e.g., \cite{chen2021regret,chen2020robust,koccyiugit2020distributionally,qi2016preferences}). For example, \cite{qi2016preferences} presented empirical evidence supporting the use of the Hurwicz criterion and showcased its potential predictive capability in path selection and traffic equilibrium. \cite{chen2021regret} provided the analysis on using the Hurwicz model to reduce over-conservatism and achieve better out-of-sample performance in the newsvendor model.
Specifically, for a given level of optimism $\bar{\lambda}\in[0,1]$, the Hurwicz model considers
\begin{align*}
 \min_{\bm x \in \X} \left\{ \bar{\lambda} \inf_{\Pr\in\P}\E_{\Pr}\left[Q(\bm x, \trxi)\right] + (1-\bar{\lambda})\sup_{\Pr\in\P}\E_{\Pr}\left[Q(\bm x, \trxi)\right]\right\}.\tag{Hurwicz Model}
\end{align*}
Moreover, the discrepancy between DRO \eqref{dro} and DFO \eqref{dfo} can be used to upper bound the regret. That is, in the face of the ambiguity set $\P$, the decision-maker chooses a solution $\bm{x}$ to minimize the distributionally robust regret, which is
\begin{align*} 
\min_{\bm x \in \X}\sup_{\Pr\in\P}\left\{E_{\Pr}\left[Q(\bm x, \trxi)\right] - \min_{\bm y \in \X}\E_{\Pr}\left[Q(\bm y, \trxi)\right]\right\}. \tag{Regret}
\end{align*} 
The regret is upper bounded by the following value
\begin{align*}
\min_{\bm x \in \X}\sup_{\Pr\in\P}\E_{\Pr}\left[Q(\bm x, \trxi)\right] -  \min_{\bm x \in \X}\inf_{\Pr\in\P}\E_{\Pr}\left[Q(\bm x, \trxi)\right],
\end{align*}
which represents the possible range of objective function values. The difference between DRO \eqref{dro} and DFO \eqref{dfo} can also be used to quantify the level of uncertainty and to bound the value of the stochastic program (see, e.g., \cite{chen2021regret,chinchilla2022stochastic}). 

While DFO \eqref{dfo} can be applied to many domains, a thorough complexity analysis of DFO \eqref{dfo} remains absent from existing works. This paper aims to bridge this gap by exploring the computational complexities of DFO \eqref{dfo} in depth. Specifically, when comparing DFO \eqref{dfo} with DRO \eqref{dro}, we see that the inner supremum operator in DRO \eqref{dro} maintains the convexity, but the inner infimum in DFO \eqref{dfo} typically undermines this convexity (see, e.g., \cite{beck2009duality}). To address this challenge, we revisit and expand upon the mixed-integer convex programming representability concept, adapting it to DFO \eqref{dfo}. Our results allow standard optimization solvers to solve DFO \eqref{dfo} efficiently. 
We expect that the mixed-integer convex programming representability results presented in this paper can be applied to decision-dependent uncertainty (see the discussions in \cite{nohadani2018optimization,goel2006class}). Throughout this paper, we make the
following assumptions of set $\X$:
\begin{assumption}\label{assum_0}
 Set $\X$ is compact and has a non-empty relative interior.
\end{assumption}
Assumption~\ref{assum_0} is standard in the stochastic optimization literature (see, e.g., \cite{shapiro2002minimax}). It is worth mentioning that Assumption~\ref{assum_0} is useful to prove that DFO may not be mixed-integer convex programming representable. We also note that the result in this paper can be directly extended to mixed-integer compact convex set $\X$, which can be written as a finite union of compact convex sets with a non-empty relative interior.


\subsection{Relevant Literature}
Recent studies on DFO address outliers and uncertainties of decision-making problems (see, e.g.,  \cite{norton2017optimistic, song2020optimistic,royset2022rockafellian, gotoh2023data,jiang2023dfo}). Notably, \cite{royset2022rockafellian} studied two-stage stochastic optimization without relatively complete recourse from the DFO perspective, using the Rockafellian relaxation for perturbation analysis to tackle model uncertainty. \cite{norton2017optimistic} applied DFO to manage noises and outliers in machine learning, while \cite{gotoh2023data} investigated the out-of-sample performance and sensitivity analysis of DFO solutions, especially those involving regularization. For a comprehensive understanding of DFO, readers are directed to the discussions in \cite{royset2022rockafellian,jiang2023dfo} and the references cited therein.

Despite its vital importance and wide applications, DFO \eqref{dfo} often faces a significant challenge: the absence of convexity, which renders it more complex to solve compared to its DRO counterpart. For example, \cite{norton2017optimistic} investigated intractable machine learning problems by employing a nonconvex regularizer based on DFO. This paper is quite different from \cite{jiang2023dfo}. The latter work provided a comprehensive introduction to the DFO framework to illustrate its suitability for decision-making, such as showing how DFO \eqref{dfo} can recover many robust statistics. However, \cite{jiang2023dfo} did not study the computational tractability of DFO \eqref{dfo}. In fact, to date, the computational tractability of DFO has not been extensively explored in the literature. To fill this gap, our paper studies the tractability and complexity of DFO \eqref{dfo}. Recognizing that strong conditions are required for the tractability of DFO \eqref{dfo}, we study the conditions that allow for its representation as a mixed-integer convex program. Our results build upon recent findings by \cite{lubin2022mixed} in mixed-integer convex programming representability, focusing on special cases of DFO that can be precisely formulated as mixed-integer convex programs. Leveraging this notion of mixed-integer convex programming representability advances our understanding of DFO and opens avenues for more efficient computational approaches.

\subsection{Contributions}

This paper complements the literature by providing the tractability and complexity analysis for DFO \eqref{dfo}. We prove that solving DFO \eqref{dfo} is NP-hard in general, and we provide conditions under which DFO \eqref{dfo} can be tractable, i.e., when the function $Q(\bm x,\rxi)$ is convex or concave piecewise affine in $\bm{x}$ in sDFO \eqref{eq_uq_simple} under mild conditions. We generalize the mixed-integer convex programming (MICP) representable (MICP-R) notion, introduced in \cite{lubin2022mixed}. Particularly, we show the sufficient conditions under which sDFO \eqref{eq_uq_simple} and DFO \eqref{dfo} can be MICP-R. Table~\ref{tab_summary} summarizes our main MICP-R results. We numerically demonstrate the value of MICP-R for DFO \eqref{dfo} and find that the MICP-R formulation can dramatically improve the running time, even for small-scale instances. The MICP-R results can be directly applied to the Hurwicz model, which can effectively reduce over-conservatism and achieve better out-of-sample performance. This is also the first-known complexity result of the Hurwicz model.

\begin{table}[htbp]
\vspace{-1.5em}
\centering
\caption{Summary of MICP-R and Not MICP-R Results}
\setlength{\tabcolsep}{1pt} 
\renewcommand{\arraystretch}{1} 
\label{tab_summary}
\begin{center}
\begin{tabular}{|c|c|c|c|}
\hline
\multirow{5}{*}{sDFO \eqref{eq_uq_simple} } & Function $Q(\bm x, \rxi)$& Uncertainty Set & Results\\ \cline{2-4} 
& \multirow{2}{*}{\makecell{Concave Piecewise Affine\\ Section~\ref{eq_uq_simple_cases_concave}}} & \multirow{4}{*}{ $\U=\{\rxi : \|\rxi-\rxi^0\|_p\leq \theta\}$}   & $p\in\{1,\infty\}:$ MICP-R \\ \cline{4-4} 
& &   & $p\in(1,\infty):$ Not MICP-R \\ \cline{2-2} \cline{4-4} 
& \multirow{2}{*}{\makecell{Convex Piecewise Affine\\Section~\ref{eq_uq_simple_cases_convex} }}  & & $p\in\{1,\infty\}$ \& Mild Conditions: MICP-R \\ \cline{4-4} 
& &   & $p\in(1,\infty):$ Not MICP-R\\ \hline
\hline
\multirow{4}{*}{DFO \eqref{dfo}}  & \multirow{4}{*}{\makecell{Function $Q(\bm x, \rxi)$ \\ Can be MICP-R}}  & Ambiguity Set   & Results\\ \cline{3-4} 
& & \multirow{2}{*}{\makecell{Type-$\infty$ Wasserstein \\ Section~\ref{DFO_Wasserstein_MICPR}}} & $p\in\{1,\infty\}$ \& Mild Conditions: MICP-R  \\ \cline{4-4} 
& & & $p\in(1,\infty):$ Not MICP-R \\ \cline{3-4} 
& & \makecell{Finite Support\\Section~\ref{sec_dfo_finite_support} } & MICP-R \\ \hline
\end{tabular}
\end{center}
\vspace{-2em}
\end{table}

\noindent\textbf{Organization.} The remainder of the paper is organized as follows. Section~\ref{sec_preliminary} reviews and extends the MICP-R notion. Section~\ref{sec_piecewise_affine} discusses the tractability analysis, complexity analysis, and MICP-R formulations for sDFO \eqref{eq_uq_simple}. Section~\ref{sec_ambiguity_set} discusses the tractability analysis, complexity analysis, and MICP-R formulations for DFO \eqref{dfo}. Section~\ref{sec_numerical} numerically demonstrates the value of MICP-R formulations for DFO \eqref{dfo}.
Section~\ref{sec_conclusion} concludes the paper.

\noindent\textbf{Notation.} The following notation is used throughout the paper. We use bold letters (e.g., $\bm{x}, \bm{A}$) to denote vectors and matrices and use corresponding non-bold letters to denote their components. 
We let $\|\cdot\|_*$ denote the dual norm of a general norm $\|\cdot\|$. We let $\bm{e}$ be the vector or matrix of all ones, and let $\bm{e}_i$ be the $i$th standard basis vector.
Given an integer $n$, we let $[n]:=\{1,2,\ldots,n\}$, and use $\Re_+^n:=\{\bm {x}\in \Re^n:x_i\geq0, \forall i\in [n]\}$. Given a real number $t$, we let $(t)_+:=\max\{t,0\}$ and $(t)_-:=\min\{t,0\}$. 
Given a set $I$, we use $\mathrm{ext}\{I\}$ to represent its extreme points. 
We let $\tilde{\bm\xi}$ denote a random vector and denote its realizations by $\bm\xi$. 
Given a probability distribution $\Pr$ defined on support $\U$ with sigma-algebra $\F$ and a $\Pr$-measurable function $g(\bm{\xi})$, we use $\Pr\{A\}$ to denote $\Pr\{\trxi:\text{condition} \ A(\trxi) \ \text{holds}\}$ when $A(\trxi)$ is a condition on $\bm\xi$, and to denote $\Pr\{\trxi\colon \trxi \in A\}$ when $A \in \F$ is $\Pr$-measurable, and we let ${\rm{ess.sup}}_{\Pr}(g(\trxi))$ denote the essential supremum of the random function $g(\trxi)$. 
Given a set $R$, the characteristic function $\chi_{R}(\bm x)=0$ if $\bm x \in R$, and $\infty$, otherwise. 
 We let $\delta_{\omega}$  denote for the Dirac distribution that places unit mass on the realization $\omega$. We use $\floor{x}$ and $\ceil{x}$ to denote the largest integer $y$ satisfying $y\leq x$ and the smallest integer $y$ satisfying $y\geq x$ for any $x\in \Re$, respectively. Additional notations will be introduced as needed.

\section{NP-hardness of DFO and the MICP-R Notion}
\label{sec_preliminary}

DFO \eqref{dfo} can be viewed as a biconvex program, which is notoriously known to be computationally challenging. Hence, in this section, we first show that solving DFO \eqref{dfo} is, in general, an NP-hard problem. Nevertheless, we are able to identify tractable DFO special cases, where in this section, we formally define the tractability. In contrast to biconvex programs, mixed-integer convex programs have recently been shown to be more scalable and capable of solving many large-scale problems (see more discussions in \cite{achterberg2013mixed,achterberg2020presolve} and the references cited therein). Therefore, we focus on studying MICP-R reformulations of DFO \eqref{dfo}, and this section formally defines the MICP-R notions as a preliminary of our main results.

 \subsection{NP-hardness and Tractability of DFO}
We observe that evaluating the most favorable objective function value of DFO \eqref{dfo} for a given decision can be NP-hard, even under a very simple setting.

\begin{restatable}{proposition}{propevadfo}\label{prop_eva_dfo} 
Computing the inner infimum of DFO \eqref{dfo}, in general, is NP-hard even when the ambiguity set $\P = \{\Pr\colon\Pr\{\trxi\in {\U}\}=1\}$ with box uncertainty set $\U$ and the recourse function $Q(\bm x,\rxi)$ only involves the objective uncertainty. 
\end{restatable}
\begin{proof} 
Let us consider the NP-complete problem --- set partition problem, which asks
\begin{quote}\it
\textbf{Set partition problem.} Given $N$ nonnegative integers $w_1, w_2, \cdots, w_N$, does there exist one set partition $S$, such that $\sum_{i\in S}w_i=\sum_{i\in[N]\setminus S} w_i$?
\end{quote}
In DFO \eqref{dfo}, let the ambiguity set $\P = \{\Pr\colon\Pr\{\trxi\in {\U_I}\}=1\}$ with an interval uncertainty set $\U_I=[-1,1]^N$,
and let 
\begin{align*}
Q(\bm x,\rxi) =\min_{\bm y\in\Y}\sum_{i\in[N]}\xi^i(\bm a_i^\top \bm y-b_i),
\end{align*}
where $\bm a_i=\bm e_i$ and $b_i=0$ for each $i\in[N]$, and  set $\Y=\{\bm y\in \Re^N: -1\leq y_i\leq 1, \forall j\in[N],\sum_{i\in[N]} w_jy_j=0\}$.
In this setting, the inner infimum of DFO \eqref{dfo} reduces to 
\begin{subequations}
\begin{align}\label{example_uq_hard}
v^*=\min_{\rxi,\bm y}\left\{\sum_{i\in[N]}\xi^iy_i\colon -1\leq \xi^i\leq 1, \forall i\in[N],  -1\leq y_i\leq 1, \forall i\in[N],\sum_{i\in[N]} w_iy_i=0 \right\}.
\end{align}
Above, optimizing over $ \rxi$ first, problem \eqref{example_uq_hard} reduces to 
\begin{align}\label{example_uq_hard_opt_y}
v^*=\min_{\bm y}\left\{ -\sum_{i\in[N]}\max\left(y_i,0\right)+\sum_{i\in[N]}\min\left(y_i,0\right)\colon  -1\leq y_i\leq 1, \forall i\in[N],\sum_{i\in[N]} w_iy_i=0 \right\}.
\end{align}
\end{subequations}
Then, we observe that the optimal value $v^*=-N$ in \eqref{example_uq_hard_opt_y} if and only if there exists an optimal solution $\bm y^*\in\{-1,1\}^N$, i.e., the optimal value $v^*=-N$ in \eqref{example_uq_hard_opt_y} if and only if there exists a set partition such that $\sum_{i\in S}w_i=\sum_{i\in[N]\setminus S} w_i$. Since the set partition problem is NP-hard, solving problem \eqref{example_uq_hard_opt_y} is NP-hard. That is, computing the inner infimum of DFO \eqref{dfo} is NP-hard.
\QEDA
\end{proof}

Proposition \ref{prop_eva_dfo} motivates us to explore special cases under which computing the inner infimum of DFO \eqref{dfo} is tractable. 
Formally, we define tractable convex programs following the convention from work \cite{ben2009robust}, as below.
\begin{definition}\label{def_tract}
(Tractability, theorem A.3.3 in \cite{ben2009robust}) Suppose that for any given compact set $\X\subseteq \Re^n$, which has a nonempty relative interior and is contained in a Euclidean ball with radius $R$ and is containing a Euclidean ball with radius $r$, then there exists an efficient algorithm to solve the favorable problem $\min_{\bm x\in\X}\inf_{\Pr\in\P} \E_{\Pr}[Q(\bm x,\trxi)]$ to $\hat\varepsilon>0$ accuracy, whose running time is polynomial in $n,m,\ln(R/r),\ln(1/\hat\varepsilon)$, and the encoding length of $\min_{\bm x\in\X}\inf_{\Pr\in\P} \E_{\Pr}[Q(\bm x,\trxi)]$.
\end{definition}

\subsection{The MICP-R Notion}
As this paper aims to explore conditions under which DFO \eqref{dfo} can be mixed-integer convex programming (MICP) representable (MICP-R), we formally define this notion, initially introduced in the work \cite{lubin2022mixed}, below.
\begin{definition}\label{def_micpr}
\begin{itemize}
\item[(i)] (definition 1.1 in \cite{lubin2022mixed}) Given $n,p,d\in\Ze_+$, suppose that sets $\S\subseteq\Re^n$ and $\M\subseteq \Re^{n+p+d}$ are closed and convex. Then the tuple $(\M,p,d)$ induces an MICP formulation of set $\S$ if 
\begin{align*}
\bm x\in\S \Leftrightarrow \exists \bm y\in \Re^p, \bm z\in\Ze^d, \textup{s.t. } (\bm x,\bm y,\bm z)\in \M;
\end{align*}
\item[(ii)] (An MICP-R Set,  definition 1.2 in \cite{lubin2022mixed}) A set $\S\in\Re^n$ is  MICP representable (MICP-R) if there exists a closed convex set $\M$ and two positive integers $p$ and $d$ that induce an MICP formulation of set $\S$; 
\item[(iii)] (An MICP-R Function) A function $f:\S\rightarrow \Re$ is MICP-R if both its domain $\S$ and its epigraph are MICP-R; and
\item[(iv)] (An MICP Formulation) A mathematical program is MICP-R if both its feasible region and objective function are MICP-R.
\end{itemize}
\end{definition}

The definition of not being MICP-R is simply the opposite of being MICP-R, which is, unfortunately, difficult to verify in practice. Fortunately, the authors in \cite{lubin2022mixed} provided a simple and sufficient condition to prove that a set is not MICP-R.
\begin{lemma}\label{lemma2_lubin}(lemma 4.1 in \cite{lubin2022mixed}) A set $\S\in\Re^n$ is not MICP-R if there exists an infinite sequence $\{\hat{\bm{x}}^j\}_j$ such that $\hat{\bm{x}}^{j_1}\neq \hat{\bm{x}}^{j_2}\in \S$ for all $j_1\neq j_2$ and $1/2(\hat{\bm{x}}^{j_1}+ \hat{\bm{x}}^{j_2})\notin \S$.
\end{lemma}
For brevity of notation, we also introduce the McCormick representation \cite{mccormick1976computability} of a simple bilinear set having a binary variable.
\begin{definition}
\label{mc_defintion}
(McCormick Representation of a Simple Bilinear Set, \cite{mccormick1976computability}) The bilinear set $\{(s,\lambda, \gamma)\in \Re\times \{\lambda_l,\lambda_u\}\times [\gamma_l,\gamma_u]:s=\lambda\gamma\}$ admits the following mixed-integer linear programming (MILP) McCormick representation:
\begin{align*}
\mathcal{MI} (\lambda_l,\lambda_u,\gamma_l,\gamma_u)
=&  \left\{(s,\lambda, \gamma)\colon \begin{aligned}
& s\in \Re, \lambda \in \{\lambda_l,\lambda_u\}, \gamma_l\leq  \gamma \leq \gamma_u,\\
& s \geq \lambda_l\gamma+\lambda\gamma_l -\lambda_l\gamma_l, s \geq \lambda_u\gamma+\lambda\gamma_u -\lambda_u\gamma_u,\\
& s \leq \lambda_u\gamma+\lambda\gamma_l -\lambda_u\gamma_l, s \leq \lambda\gamma_u+\lambda_l\gamma -\lambda_l\gamma_u 
\end{aligned}
\right\},
\end{align*}
where $\lambda_l,\lambda_u$ and $\gamma_l,\gamma_u$ are the known lower and upper bounds for $\lambda$ and $\gamma$, respectively.
\end{definition}
According to Definition~\ref{def_micpr} and Definition~\ref{mc_defintion}, the following result shows that the reverse norm function $f(\bm{x})=-\|\bm x\|_{p}+\chi_{\X}(\bm{x})$ can be either MICP-R or not MICP-R, which depends on the norm $-\|\cdot\|_{p}$ (recall that set $\X$ is compact and has a nonempty relative interior based on Assumption~\ref{assum_0}).

\begin{restatable}{lemma}{micpexample}\label{micp_example} 
The reverse norm function $f(\bm{x})=-\|\bm x\|_{p}+\chi_{\X}(\bm{x})$ is MICP-R if $p\in \{1,\infty\}$
and is not MICP-R if $p\in(1,\infty)$.
\end{restatable}

\begin{subequations}
\begin{proof}
We focus on the MICP-R formulation of the epigraph of the function $f(\cdot)$, which reads as
\begin{align}
\epi(f)=\left\{(\bm x,t)\colon-\|\bm x\|_{p} \leq t, \bm x\in\X\right\}. \label{epigraph_set_micp_lemma}
\end{align}
Next, we split the proof into three cases based on the choice of $p$. 

\noindent\textbf{Case 1: When $p=1$, i.e., the norm is $L_1$,}  we have 
\begin{align*}
\epi(f)=\left\{(\bm x,t)\colon \sum_{i\in[n]}|x_i|\geq -t,\bm x\in\X \right\},
\end{align*}
which is equivalent to
\begin{align*}
\epi(f)=\left\{(\bm x,t)\colon \max_{\bm{z}\in\{-1,1\}^n}\sum_{i\in[n]}x_iz_i\geq -t,\bm x\in\X \right\},
\end{align*}
or 
\begin{align*}
\epi(f)=\left\{(\bm x,t)\colon \sum_{i\in[n]}x_iz_i\geq -t,\bm x\in\X,\bm{z}\in \{-1,1\}^n \right\}.
\end{align*}
Since set $\X$ is compact, we can assume that $\X\subseteq [\bm{l},\bm{u}]$, i.e., given $\bm{x}\in \X$, we have $x_i\in[l_i,u_i]$ for each $i\in[n]$. We can apply the following McCormick inequalities (see more details in Definition \ref{mc_defintion}) to linearize the bilinear term $\{s_i:=x_iz_i\}_{i\in[n]}$, i.e., $(s_i,z_i,x_i)\in \mathcal{MI}(-1,1,l_i,u_i) $ for each $i\in[n]$.
Thus, $\epi(f)$ is MICP-R, i.e.,
\begin{align*}
\epi(f)=\left\{(\bm x,t)\colon \begin{array}{l}
\displaystyle \exists \bm{s}\in \Re^n, \sum_{i\in[n]}s_i\geq -t,\bm x\in\X,\bm{z}\in \{-1,1\}^n, \\
\displaystyle(s_i,z_i,x_i)\in \mathcal{MI}\left(-1,1,l_i,u_i\right), \forall i \in [n] \end{array} \right\}.
\end{align*}

\noindent\textbf{Case 2: When $p\in(1,\infty)$, i.e., the norm is neither $L_1$ nor $L_\infty$,} 
since set $\X$ is compact and has a nonempty relative interior, there exists an open ball $B(\bar{\bm{x}},r)$ centered at $\bar{\bm{x}}$ and a positive radius $r>0$ such that the intersection of $B(\bar{\bm{x}},r)$ and the affine space of set $\X$ is contained in set $\X$. Therefore, set $$\S:=\left\{\bm{x}\in X:\|\bm{x}\|_p=\|\bar{\bm{x}}\|_p\right\}$$ has a nonempty relative interior. Thus, we can pick a sequence of distinct elements from set $\S$ (e.g., all the possible rational elements) $\{\hat{\bm{x}}^j\}_j$. Since $(\hat{\bm{x}}^j,\|\bar{\bm{x}}\|_p)\in \epi(f)$ for each $j$ and function $\|\bm x\|_{p}$ is strictly convex for any $p\in (1,\infty)$, for any pair $(j_1,j_2)$ with $j_1\neq j_2$, we must have
\[\frac{1}{2}(\hat{\bm{x}}^{j_1},\|\bar{\bm{x}}\|_p)+\frac{1}{2}(\hat{\bm{x}}^{j_2},\|\bar{\bm{x}}\|_p) \notin \epi(f).\]
Since $\{\hat{\bm{x}}^j\}_j$ is an infinite sequence, according to Lemma~\ref{lemma2_lubin}, set $\epi(f)$ is not MICP-R. 

\noindent\textbf{Case 3: When $p=\infty$, i.e., the norm is $L_\infty$,} set $\epi(f)$ reduces to
\begin{align*}
\epi(f)=\left\{(\bm x,t)\colon  \max_{i\in [n]}|x_i| \geq -t, \forall i\in[n],\bm x\in\X\right\},
\end{align*}
which can be reformulated in the form of the following disjunction \cite{balas1979disjunctive}:
\begin{align*}
\epi(f)=\bigvee_{i\in [n]}\left\{(\bm x,t)\colon x_i\geq -t,\bm x\in\X\right\}\bigvee_{i\in [n]}\left\{(\bm x,t)\colon -x_i\geq -t,\bm x\in\X\right\}.
\end{align*}
Since set $\X$ is compact, the MICP-R formulation of set $\epi(f)$ follows the well-known results from disjunctive programming \cite{balas1979disjunctive}.
\QEDA
\end{proof}
\end{subequations}

Note that when $p=\infty$, although being MICP-R, the optimization of the function $f(\bm x)$ can be done efficiently by solving $2n$ tractable convex programs. The result in Lemma~\ref{micp_example} is useful to prove that many DFO problems can be MICP-R. Moreover, it is worth mentioning that Lemma~\ref{micp_example} is also applicable in analyzing the MICP-R formulation of many distributional robust optimization problems. For example, \cite{jiang2023also} discussed the MICP-R formulation in distributionally robust chance constrained programs under the general Wasserstein ambiguity set.

\section{sDFO with Piecewise Affine Functions}
\label{sec_piecewise_affine}
Motivated from Section~\ref{sec_preliminary}, we provide sufficient conditions for DFO \eqref{dfo} to be tractable or MICP-R. In this section, we first explore sDFO \eqref{eq_uq_simple} as a special case of DFO \eqref{dfo}. Similar to works in robust optimization literature \cite{ben2009robust,esfahani2018data,xie2020tractable}, we focus on the function $Q(\bm x,\rxi)$ being convex or concave piecewise affine in $\bm{x}$, respectively.

\subsection{sDFO \eqref{eq_uq_simple} with Concave Piecewise Affine Functions}
\label{eq_uq_simple_cases_concave}
We first consider that the 
function $Q(\bm x, \rxi)$ is the minimum of $K$ piecewise affine functions $\rxi^\top \bm a_k(\bm x)+b_k(\bm x)$ with affine mappings $\bm a_k(\bm x)=\hat{\bm{A}}_k \bm{x}+\hat{\bm a}_k\in \Re^{m}$ with $\hat{\bm{A}}_k\in \Re^{m\times n},\hat{\bm a}_k\in \Re^m$ and $b_k(\bm x)=\hat{\bm{B}}_k^\top \bm{x}+\hat{b}_k\in \Re$ with $\hat{\bm{B}}_k\in \Re^{ n},\hat{ b}_k\in \Re$ for each $k\in[K]$, i.e., $Q(\bm x, \rxi)=\min_{k\in[K]}[\rxi^\top \bm a_k(\bm x)+b_k(\bm x)]$. Suppose that the uncertainty set is defined as the ball $\U=\{\rxi : \|\rxi-\rxi^0\|_p\leq \theta\}$ with the known parameter $\rxi^0$ and the radius $\theta\geq 0$. It is worth mentioning that the reformulations and complexity analyses can be simply extended to more general uncertainty sets, such as polyhedral and ellipsoidal (see, e.g., \cite{ben2009robust}). In this subsection, for brevity, we focus on the ball uncertainty set $\U$. In this setting, sDFO \eqref{eq_uq_simple} becomes
\begin{align}
v^*=\min_{\bm x\in\X}\min_{\rxi\in\U} \min_{k\in[K]}  \left\{\rxi^\top \bm a_k(\bm x)+b_k(\bm x)\right\}. \label{robust_max_b_orginial}
\end{align}
Switching the first minimum operator with the third one and invoking the definition of dual norm, problem \eqref{robust_max_b_orginial} is further equivalent to
\begin{align*}
v^*=\min_{k\in [K]}\min_{\bm x\in \X}\left\{{\rxi^0}^\top \bm a_k(\bm x) + b_k(\bm x) - \theta \|\bm a_k(\bm x) \|_{p^*}\right\},
\end{align*}
which can be solved by selecting the lowest objective value within these $K$ mathematical programs, that is,
\begin{align}
v^*=\min_{k\in [K]}\left\{v_k^*:=\min_{\bm x\in \X}\left\{{\rxi^0}^\top \bm a_k(\bm x) + b_k(\bm x) - \theta \|\bm a_k(\bm x) \|_{p^*}\right\}\right\}. \label{robust_min_b}
\end{align}
Note that the inner minimization of sDFO \eqref{robust_min_b} is a concave minimization problem and, in general, can be difficult. However, by exploring the properties of the dual norm,  there are some conditions under which sDFO \eqref{robust_min_b} can be tractable.
\begin{restatable}{theorem}{proprobustmintractable}\label{prop_robust_min_tractable} 
sDFO \eqref{robust_min_b} can be tractable if either condition holds:
\begin{itemize}
\item[(i)] If $ \|\bm a_k(\bm x) \|_{p^*}:=C_k$ is constant for each $k\in[K]$ and $\bm{x}\in \X$, sDFO \eqref{robust_min_b} is equivalent to solving $K$ tractable convex programs, i.e., $v^*=\min_{k\in [K]}v_k^*$, where for each $k\in [K]$, we have
\begin{align*}
v_k^*=\min_{\bm x\in\X}\left\{{\rxi^0}^\top \bm a_k(\bm x) + b_k(\bm x) - \theta C_k\right\};
\end{align*}  
\item[(ii)] If $p=1$, sDFO \eqref{robust_min_b} is equivalent to solving $2mK$ tractable convex programs, i.e., $v^*=\min_{k\in [K],i\in[m],\ell\in [2]}v_{ik\ell}^*$, where for each $k\in [K]$ and $i\in[m]$, we have
\begin{align*}
v_{ik1}^*=\min_{\bm x\in\X}\left\{{\rxi^0}^\top \bm a_k(\bm x)+ b_k(\bm x) + \theta a_{ki}(\bm x)\right\},\quad v_{ik2}^*=\min_{\bm x\in\X}\left\{{\rxi^0}^\top \bm a_k(\bm x)+ b_k(\bm x) -\theta a_{ki}(\bm x)\right\}.
\end{align*}
\end{itemize}
\end{restatable}
\begin{proof}
We split the proof into two parts by checking these two conditions separately. 
\begin{itemize}
\item[(i)] When $ \|\bm a_k(\bm x) \|_{p^*}$ is a constant for each $k\in[K]$, i.e.,  $\|\bm a_k(\bm x) \|_{p^*}=C_k$ for $k\in[K]$, then the objective function of sDFO \eqref{robust_min_b} is linear and optimizing it is equivalent to solving $K$ convex programs;
\item[(ii)] When $p=1$, i.e., when the dual norm is $L_\infty$, 
then
$\theta \|\bm a_k(\bm x) \|_\infty =  \theta \max_{i\in [m]}\max\left\{ a_{ki}(\bm x),-a_{ki}(\bm x)\right\}$. That is, sDFO \eqref{robust_min_b} can be simplified as
\begin{align*}
\min_{k\in [K]}\min_{i\in [m]}\min\left\{ \min_{\bm x\in\X}\left\{{\rxi^0}^\top \bm a_k(\bm x) + b_k(\bm x)+ \theta a_{ki}(\bm x)\right\},\min_{\bm x\in\X}\left\{{\rxi^0}^\top \bm a_k(\bm x) + b_k(\bm x)-\theta a_{ki}(\bm x)\right\}\right\}, 
\end{align*}
which is equivalent to solving $2mK$ convex programs and selecting the best one with the lowest optimal value. \QEDA
\end{itemize}
\end{proof}

In the following complexity analysis, we focus on the non-trivial cases where $ \|\bm a_k(\bm x) \|_{p^*}$  is not a constant for some $k\in [K]$. Unfortunately, when $p\in(1,\infty]$, solving sDFO \eqref{robust_min_b}, in general, is NP-hard with the reduction to the well-known NP-hard problem --- maximizing a norm over a polytope. 

\begin{restatable}{proposition}{propnphardrobustminb}\label{prop_np_hard_robust_min_b} 
For any $p\in(1,\infty]$, solving sDFO \eqref{robust_min_b}, in general, is NP-hard even with $K=1$.
\end{restatable}
\begin{proof}
Let us consider an NP-hard problem --- Norm maximization over a polytope (see theorem 1 in \cite{ge2011note}), which asks
\begin{quote}\it
\textbf{Norm maximization over a polytope.} Given the polytope $\left\{\bm x:\bm D\bm x\leq \bm d\right\}$, where $\bm D \in \Re^{\tau\times n}$ and $\bm d\in\Re^\tau$, what is the optimal value of the problem $\max_{\bm x}\left\{ \|\bm x\|_{p^*}\colon \bm D\bm x\leq \bm d \right\}$ with $p\in(1,\infty]$?
\end{quote}
Consider a special case of sDFO \eqref{robust_min_b}, where $K=1$, $\bm a_1(\bm x)=\bm x$, $b_1(\bm x)=0$,
 $\bm \xi^0=\bm 0$, $\theta=1$, and set $\X=\{\bm x: \bm D\bm x\leq \bm d\}$. In this case, sDFO \eqref{robust_min_b} can be written as  
\begin{align*}
\max_{\bm x}\left\{ \|\bm x \|_{p^*}\colon \bm D\bm x\leq \bm d \right\},
\end{align*}
which is exactly the desirable norm maximization problem over a polytope for any $p\in(1,\infty]$. Thus, solving sDFO \eqref{robust_min_b}, in general, is NP-hard for any $p\in(1,\infty]$.  
\QEDA
\end{proof}

The complexity result suggests that the tractable result in Theorem~\ref{prop_robust_min_tractable} with $p=1$ is the best one that we could expect.

Next, for the intractable case, we study the MICP-R formulation of the objective function of sDFO  \eqref{robust_min_b}.
As an extension of Lemma~\ref{micp_example}, we notice that when $p=\infty$, the objective function of sDFO \eqref{robust_min_b} is MICP-R; otherwise, when $p\in(1,\infty)$, it is not. 
\begin{restatable}{theorem}{propminmicpr}\label{prop_min_micpr} 
When $p=\infty$, the objective function of sDFO \eqref{robust_min_b} with domain $\X$ is MICP-R; otherwise, when $p\in(1,\infty)$, the objective function of sDFO \eqref{robust_min_b} with domain $\X$ may not be MICP-R.
\end{restatable}
\begin{proof}
We first rewrite the objective function of sDFO \eqref{robust_min_b} with domain $\X$ as $f_k(\bm{x})={\rxi^0}^\top \bm a_k(\bm x) + b_k(\bm x)-\|\bm a_k(\bm x)\|_{p^*}+\chi_{\X}(\bm{x})$ for each $k\in[K]$. We then focus on the MICP-R formulation of the epigraph of the function $f_k(\cdot)$, which reads as
\begin{align}
\epi(f_k)=\left\{(\bm x,t)\colon{\rxi^0}^\top \bm a_k(\bm x) + b_k(\bm x)-\|\bm a_k(\bm x) \|_{p^*}\leq t, \bm x\in\X\right\}. \label{robust_min_micp}
\end{align}
Next, we split the proof into two cases based on the choice of $p$. 


\noindent\textbf{Case 1: When $p=\infty$, i.e., the dual norm is $L_1$,} set $\epi(f_k)$ can be written as
\begin{align*}
\epi(f_k)=\left\{(\bm x,t)\colon \sum_{i\in[m]}\left|a_{ki}(\bm x)\right|-{\rxi^0}^\top \bm a_k(\bm x) - b_k(\bm x)\geq -t,\bm x\in\X \right\}.
\end{align*}
Then we have
\begin{align*}
\epi(f_k)=\left\{(\bm x,t)\colon \max_{\bm{z}\in\{-1,1\}^m}\sum_{i\in[m]}a_{ki}(\bm x)z_i-{\rxi^0}^\top \bm a_k(\bm x) - b_k(\bm x)\geq -t,\bm x\in\X \right\},
\end{align*}
which is equivalent to
\begin{align*}
\epi(f_k)=\left\{(\bm x,t)\colon \sum_{i\in[m]}a_{ki}(\bm x)z_i-{\rxi^0}^\top \bm a_k(\bm x) - b_k(\bm x)\geq -t,\bm x\in\X,\bm{z}\in \{-1,1\}^m \right\}.
\end{align*}
Since set $\X$ is compact, we can apply the McCormick inequalities (see Definition~\ref{mc_defintion}) to linearize the bilinear terms $\{a_{ki}(\bm x)z_i\}_{i\in[m]}$. Thus, $ \epi(f_k)$ is MICP-R for each $k\in[K]$.

\noindent\textbf{Case 2: When $p\in(1,\infty)$, i.e., when the dual norm is neither $L_1$ nor $L_\infty$,} suppose that $K=1$, $\rxi^0=\bm 0$, $\bm a_1(\bm x)=\bm x$, and $b_1(\bm x)=0$, then set \eqref{robust_min_micp} reduces to 
\begin{align*}
\epi(f_1)=\left\{(\bm x,t)\colon-\|\bm x \|_{p^*}\leq t, \bm x\in\X\right\},
\end{align*}
which is identical to \eqref{epigraph_set_micp_lemma}. According to the result in Lemma~\ref{micp_example}, when $p\in(1,\infty)$, the objective function of sDFO \eqref{robust_min_b} with domain $\X$, in general, may not be MICP-R.
\QEDA
\end{proof}

Theorem~\ref{prop_min_micpr} suggests that the objective function of sDFO  \eqref{robust_min_b} with domain $\X$ may not be MICP-R with a general norm, but it is MICP-R when the norm is $L_{\infty}$. As a direct corollary of Theorem~\ref{prop_min_micpr}, when $p=\infty$, the MICP-R formulation of sDFO  \eqref{robust_min_b} can be summarized as follows. 
\begin{corollary}
When $p=\infty$, suppose that $\X\subseteq [\bm{l},\bm{u}]$ and let $\hat{l}_{ki}=\sum_{j\in [n]}\min\{\hat{A}_{kij} l_j,\hat{A}_{kij} u_j\}+\hat{a}_{ki}$ and $\hat{u}_{ki}=\sum_{j\in [n]}\max\{\hat{A}_{kij} l_j,\hat{A}_{kij} u_j\}+\hat{a}_{ki}$ for each $i\in [m]$ such that $\bm{a}_k(\bm x)\in [\hat{\bm{l}}_k,\hat{\bm{u}}_k]$ for each $k\in[K]$. Then, sDFO \eqref{robust_min_b} is equivalent to solving the following $K$ MICPs, i.e.,
$v^*=\min_{k\in [K]}v_k^*$, where for each $k\in [K]$, we have
\begin{align*}
v_k^* = \min_{\bm x\in\X} \left\{{\rxi^0}^\top \bm a_k(\bm x) + b_k(\bm x) - \theta \sum_{i\in[m]}s_{ki}: \left(s_{ki},z_{ki},a_{ki}(\bm x)\right)\in \mathcal{MI} \left(-1,1,\hat{l}_{ki},\hat{u}_{ki}\right),\forall i\in[m]\right\}.
\end{align*}
\end{corollary}

\subsection{sDFO \eqref{eq_uq_simple} with Convex Piecewise Affine Functions}
\label{eq_uq_simple_cases_convex}

In this subsection, we follow the same notation and uncertainty set as the previous subsection (Section~\ref{eq_uq_simple_cases_concave}) and consider the maximum of piecewise affine function, that is, the function $Q(\bm x, \rxi)$ is defined as the maximum of $K$ piecewise affine function $\rxi^\top \bm a_k(\bm x)+b_k(\bm x)$, i.e., $Q(\bm x, \rxi)=\max_{k\in[K]}[\rxi^\top \bm a_k(\bm x)+b_k(\bm x)]$.  
In this setting, sDFO \eqref{eq_uq_simple} can be recast as
\begin{align}
v^* = \min_{\bm x\in\X}\min_{\rxi\in\U} \max_{k\in[K]}  \left\{\rxi^\top \bm a_k(\bm x)+b_k(\bm x)\right\}. \label{robust_max_b}
\end{align}
Let us first provide an equivalent reformulation of sDFO \eqref{robust_max_b}, which helps establish the tractability and MICP formulation of sDFO \eqref{robust_max_b}. 
\begin{restatable}{lemma}{robustmaxrefor}\label{robust_max_refor} 
sDFO \eqref{robust_max_b} is equivalent to 
\begin{align}
v^*=\min_{\bm x\in\X}\max_{\bm\lambda\geq  \bm 0}\left\{\sum_{k\in[K]} \lambda_k \left[{\rxi^0}^\top \bm a_k(\bm x) +b_k(\bm x) \right] -  \theta \left\|\sum_{k\in[K]} \lambda_k \bm a_k(\bm x) \right\|_{p^*}  \colon  \sum_{k\in[K]} \lambda_k =1 \right\}. \label{robust_max_eq}
\end{align}
\end{restatable}
\begin{proof}
Let us first consider the inner minimax of sDFO \eqref{robust_max_b} as
\begin{align*}
\min_{\rxi} \left\{\max_{k\in[K]}  \rxi^\top \bm a_k(\bm x)+b_k(\bm x)\colon \|\bm \xi-\bm \xi^0\|_p\leq \theta\right\}.
\end{align*}
Introducing auxiliary nonnegative variables $\bm\lambda$, the inner minimax of sDFO \eqref{robust_max_b} is equivalent to
\begin{align*}
\min_{\rxi}\left\{\max_{\bm\lambda\geq  \bm 0}\left\{\sum_{k\in[K]} \lambda_k \left[{\rxi}^\top \bm a_k(\bm x)+b_k(\bm x)\right]\colon  \sum_{k\in[K]} \lambda_k =1\right\}: \|\bm \xi-\bm \xi^0\|_p\leq \theta\right\}.
\end{align*} 
According to Sion's minimax theorem \cite{sion1958general}, we can interchange the maximum operator with the minimum one as
\begin{align*}
\max_{\bm\lambda\geq  \bm 0}\left\{\min_{\rxi}\left\{\sum_{k\in[K]} \lambda_k \left[{\rxi}^\top \bm a_k(\bm x)+b_k(\bm x)\right]\colon \|\bm \xi-\bm \xi^0\|_p\leq \theta  \right\}:\sum_{k\in[K]} \lambda_k =1\right\}.
\end{align*}
Invoking the definition of dual norm, we have 
\begin{align*}
\max_{\bm\lambda\geq  \bm 0}\left\{\sum_{k\in[K]} \lambda_k \left[{\rxi^0}^\top \bm a_k(\bm x) +b_k(\bm x) \right] -  \theta \left\|\sum_{k\in[K]} \lambda_k \bm a_k(\bm x) \right\|_{p^*}  \colon  \sum_{k\in[K]} \lambda_k =1 \right\}.
\end{align*}
This completes the proof.
\QEDA
\end{proof}

Due to the bilinear terms in the reformulation \eqref{robust_max_eq}, sDFO \eqref{robust_max_b}, in general, can be difficult to solve. However, by exploring the objective function and the properties of the dual norm, we are able to prove conditions under which sDFO \eqref{robust_max_b} can be tractable.
\begin{restatable}{theorem}{proprobustmaxtractable}\label{prop_robust_max_tractable} 
sDFO \eqref{robust_max_b} can be tractable if any of the following conditions holds:
\begin{itemize}
\item[(i)] When $ \bm a_k(\bm x):=\bar{\bm a}_k$ is constant for all $k\in[K]$ and $\bm{x}\in \X$,  sDFO \eqref{robust_max_b} is equivalent to
\begin{align*}
v^*=\min_{\bm x\in\X,\rxi,\eta}  \left\{\eta\colon\eta\geq \rxi^\top \bar{\bm a}_k+b_k(\bm x), \forall k\in[K], \|\rxi-\rxi^0\|_p\leq \theta\right\};
\end{align*}
\item[(ii)] When $p=1$ and  $a_k(\bm x)=a_1(\bm x)$ for each $k\in[K]$, sDFO  \eqref{robust_max_b} is equivalent to solving $2m$ tractable convex programs, and selecting the lowest optimal value, i.e., $v^*=\min_{i\in[m],\ell\in [2]}v_{i\ell}^*$, where for each $i\in[m]$, we have
\begin{align*}
& v_{i1}^*=\min_{\bm x\in\X,\eta}  \left\{\eta\colon \eta\geq  b_k(\bm x)+ {\rxi^0}^\top \bm a_1(\bm x)- \theta a_{1i}(\bm x),\forall k\in[K] \right\},\\
& v_{i2}^*=\min_{\bm x\in\X,\eta}  \left\{\eta\colon \eta\geq  b_k(\bm x)+ {\rxi^0}^\top \bm a_1(\bm x) + \theta a_{1i}(\bm x),\forall k\in[K] \right\};
\end{align*}
\item[(iii)] Suppose that $p=1$,  and $\rxi:=[\rxi_1,\ldots,\rxi_K]$ such that $\rxi^i$ and $\rxi^j$ do not overlap for each $i\neq j$, and $\rxi^\top \bm a_k(\bm x)= {\rxi^k}^\top \bar{\bm a}_{k}(\bm x) $ for each $k\in[K]$ such that $\|\bar{\bm a}_k(\bm x) \|_{\infty}=C_k$ is constant for each $k\in[K]$, where $\rxi_k\in \Re^{m_k}$, $\bar{\bm a}_k(\bm x)=\bar{\bm{A}}_k \bm{x}+\bar{\bm a}_k\in \Re^{m_k}$ with $\bar{\bm{A}}_k\in \Re^{m_k\times n},\bar{\bm a}_k\in \Re^{m_k}$ such that $\bar{\bm{A}}_i$ and $\bar{\bm{A}}_j$, $\bar{\bm{a}}_i$ and $\bar{\bm{a}}_j$ do not overlap each $i\neq j$ with $\sum_{k\in[K]}m_k=m$ and each $m_k$ is nonnegative. 
Then,  sDFO \eqref{robust_max_b} is equivalent to solving
\begin{align*}
v^*=\min_{\bm x\in\X,\beta,\bm\gamma\geq \bm 0} \left\{\beta\colon \sum_{k\in[K]} \gamma_{k}= \theta,\beta \geq  {\rxi_k^0}^\top \bar{\bm a}_{k}(\bm x)+b_k(\bm x)-\gamma_{k}C_k,\forall k\in[K]  \right\}.
\end{align*}
\end{itemize}
\end{restatable}
\begin{proof}
We split the proof into three parts accordingly.
\begin{itemize}
\item[(i)] When $\bm a_k(\bm x)=\bar{\bm a}_k$ is constant for all $k\in[K]$, sDFO \eqref{robust_max_b} can be written as
\begin{align*}
v^*=\min_{\bm x\in\X} \min_{\rxi} \max_{k\in[K]}  \left\{\rxi^\top \bar{\bm a}_k+b_k(\bm x)\colon \|\rxi-\rxi^0\|_p\leq \theta\right\}.
\end{align*}
Introducing an auxiliary variable $\eta$ to linearize the inner maximum, we arrive at Part (i).
\item[(ii)] When $\bm a_k(\bm x)=\bm a_1(\bm x)$  for each $k\in[K]$, we rewrite sDFO  \eqref{robust_max_eq} as
\begin{align*}
v^*=\min_{\bm x\in\X} \max_{\bm\lambda\geq  \bm 0} \left\{\sum_{k\in[K]} \lambda_k b_k(\bm x)+ {\rxi^0}^\top \bm a_1(\bm x) - \theta \| \bm a_1(\bm x) \|_{p^*}\colon  \sum_{k\in[K]} \lambda_k =1  \right\}.
\end{align*}
Taking the dual of the inner maximization problem and using strong duality from linear programming, we have 
\begin{align*}
v^*=\min_{\bm x\in\X,\eta}  \left\{\eta\colon \eta\geq  b_k(\bm x)+ {\rxi^0}^\top \bm a_1(\bm x)- \theta \| \bm a_1(\bm x) \|_{p^*}, \forall k\in[K] \right\}.
\end{align*}
When $p=1$, i.e., when the dual norm is $L_\infty$, 
then $\theta \|\bm a_1(\bm x) \|_\infty =  \theta \max_{i\in [m]}\max\{ a_{1i}(\bm x),-a_{1i}(\bm x)\}$. Therefore, sDFO \eqref{robust_max_eq} is equivalent to solving $2m$ tractable problems, and selecting the best one with the lowest optimal value, i.e., $v^*=\min_{i\in[m],\ell\in [2]}v_{i\ell}^*$, where for each $i\in[m]$, we have
\begin{align*}
& v_{i1}^*=\min_{\bm x\in\X,\eta}  \left\{\eta\colon \eta\geq  b_k(\bm x)+ {\rxi^0}^\top \bm a_1(\bm x)- \theta a_{1i}(\bm x),\forall k\in[K] \right\},\\
& v_{i2}^*=\min_{\bm x\in\X,\eta}  \left\{\eta\colon \eta\geq  b_k(\bm x)+ {\rxi^0}^\top \bm a_1(\bm x) + \theta a_{1i}(\bm x),\forall k\in[K] \right\}.
\end{align*}
\item[(iii)] Since $\rxi^i$ and $\rxi^j$ do not overlap for each $i\neq j$ and $\bar{\bm a}_{j}(\bm x) $ and $\bar{\bm a}_{j}(\bm x) $ do not overlap for each $i\neq j$ as well, when $p=1$ and $\|\bar{\bm a}_k(\bm x) \|_{\infty}=C_k$ is constant for each $k\in[K]$, the dual norm term in \eqref{robust_max_eq} can be simplified as
\begin{align*}
\left\|\sum_{k\in[K]} \lambda_k \bm a_k(\bm x) \right\|_{\infty} = \max_{k\in[K]} \lambda_k\left\|\bar{\bm a}_k(\bm x) \right\|_{\infty}=\max_{k\in[K]} \lambda_k C_k.
\end{align*}
Then, sDFO \eqref{robust_max_eq} can be written as
\begin{align*}
 v^*=  \min_{\bm x\in\X}\max_{\bm\lambda\geq  \bm 0}\left\{\sum_{k\in[K]} \lambda_k \left[{\rxi_k^0}^\top \bar{\bm a}_{k}(\bm x) \right] + \sum_{k\in[K]} \lambda_k  b_k(\bm x) -  \theta \max_{k\in[K]} \lambda_kC_k \colon  \sum_{k\in[K]} \lambda_k =1 \right\}.
\end{align*}
Introducing one variable $\eta$ to linearize the term $\max_{k\in[K]} \lambda_kC_k$, then we have
\begin{align*}
\max_{\bm\lambda\geq  \bm 0,\eta}\left\{\sum_{k\in[K]} \lambda_k \left[{\rxi_k^0}^\top \bar{\bm a}_{k}(\bm x) \right] + \sum_{k\in[K]} \lambda_k  b_k(\bm x) -  \theta \eta  \colon  \sum_{k\in[K]} \lambda_k =1,   \lambda_kC_k-\eta\leq 0, \forall k\in[K] \right\}. 
\end{align*}
Taking the dual of the inner maximization problem with dual variables $\beta,\bm\gamma$ and using strong duality from linear programming, we have 
\begin{align*}
\min_{\beta,\bm\gamma\geq \bm 0} \left\{\beta\colon \sum_{k\in[K]} \gamma_{k}= \theta,\beta \geq  {\rxi_k^0}^\top \bar{\bm a}_{k}(\bm x)+b_k(\bm x)-\gamma_{k}C_k,\forall k\in[K]  \right\}.
\end{align*}
This completes the proof.\QEDA
\end{itemize}
\end{proof}

The results in Theorem~\ref{prop_robust_max_tractable} may be the best ones that we could expect. 
In general, solving sDFO \eqref{robust_max_b} is NP-hard for any convex $L_p$ norm.
\begin{restatable}{proposition}{robustmaxbhard}\label{robust_max_b_hard} 
For any $p\in[1,\infty]$, solving sDFO \eqref{robust_max_b},  in general, is NP-hard.
\end{restatable}
\begin{proof}
Note that when $K=1$, the sDFO  \eqref{robust_max_b} is equivalent to formulation \eqref{robust_min_b}. Thus, the complexity results in Proposition~\ref{prop_np_hard_robust_min_b} hold, i.e., solving the sDFO  \eqref{robust_max_b} is, in general, NP-hard for $p\in(1,\infty]$. 

It remains to show that solving the sDFO  \eqref{robust_max_b} is also NP-hard when $p=1$. Let us consider the NP-complete problem --- feasibility problem of a general binary program, which asks
\begin{quote}\it  
\textbf{Feasibility of a binary program.} Given an integer matrix $\bm D \in \Ze^{\tau\times K}$, and integer vector $\bm d\in\Ze^\tau$, is there a vector $\bm x\in\{-1,1\}^K$ such that $\bm D\bm x\leq \bm d$? 
\end{quote}
Let us consider the following special case of the sDFO  \eqref{robust_max_b}. We first suppose $\theta=1$ and $\rxi^0=\bm 0$, then the uncertainty set becomes
\begin{align*}
\U = \left\{\rxi\in \Re^K\colon \|\rxi\|_1\leq 1\right\}.
\end{align*}
Next, let us consider the following function
\begin{align*}
Q(\bm x, \rxi)= \max_{k\in [K]}\max \left\{{\xi_k^k} x_k-1,-{\xi_k^k}x_k+1 \right\},
\end{align*}
and the  set $\X=\{\bm x: \bm D\bm x\leq \bm d, -1\leq x_k \leq 1,\forall k\in[K]\}$. Under this special setting, the sDFO  \eqref{robust_max_b} reduces to 
\begin{align}
v^* = \min_{\bm{x},\bm{\xi} }\left\{ \max_{k\in[K]} \left\{ |{\xi_k^k} x_k -1 |\right\}: \bm D\bm x\leq \bm d, 
\bm x\in[-1,1]^K,
\rxi^k\in[-1,1]^K,\forall k\in [K]\right\}.  \label{infty_norm_special_case}
\end{align}
We observe that the optimal value  $v^*=0$ in \eqref{infty_norm_special_case} if and only if ${\xi_k^k} x_k=1$ for all $k\in [K]$, i.e., if and only if there exists a binary feasible solution $\bm x\in\{-1,1\}^K$ such that  $\bm D\bm x\leq \bm d$. Thus, solving problem \eqref{infty_norm_special_case} is NP-hard, so is the sDFO  \eqref{robust_max_b}.
\QEDA
\end{proof} 

Proposition \ref{robust_max_b_hard} motivates us to investigate the MICP-R formulation of the objective function of sDFO \eqref{robust_max_b} with domain $\X$. Unfortunately, in most cases,
the objective function of sDFO \eqref{robust_max_b}  with domain $\X$ may not be MICP-R. 
\begin{restatable}{proposition}{propmaxmicpr}\label{prop_max_micpr} 
For any $p\in(1,\infty)$,  the objective function of sDFO \eqref{robust_max_b} with domain $\X$ may not be MICP-R. 
\end{restatable}
\begin{proof}
Note that when $K=1$, sDFO \eqref{robust_max_b} is equivalent to sDFO \eqref{robust_min_b}. Therefore, according to the result in Theorem~\ref{prop_min_micpr}, when $p\in(1,\infty)$, the objective function of sDFO \eqref{robust_max_b} with domain $\X$  may not be MICP-R. 
\QEDA
\end{proof}

When $p\notin (1,\infty)$, 
in Theorem~\ref{prop_robust_max_tractable}, we show that when $p=1$, there exist special cases such that sDFO \eqref{robust_max_b} can be tractable. Next, we show special cases under which the objective function of sDFO \eqref{robust_max_b} with domain $\X$ can be MICP-R when $p=\infty$.

\begin{restatable}{theorem}{propmaxmicprgood}\label{prop_max_micpr_good} 
When $p=\infty$, the objective function of sDFO \eqref{robust_max_b} with domain $\X$ is MICP-R if one of the following conditions holds:
\begin{itemize}
\item[(i)] When  $\bm a_k(\bm x)=\bm a_1(\bm x)$ for each $k\in[K]$; or
\item[(ii)]  When $\rxi:=[\rxi_1,\ldots,\rxi_K]$ such that $\rxi_i$ and $\rxi_j$ do not overlap for each $i\neq j$, and $\rxi^\top \bm a_k(\bm x)= {\rxi_k}^\top \bar{\bm a}_{k}(\bm x) $ for each $k\in[K]$, where $\rxi_k\in \Re^{m_k}$, $\bar{\bm a}_k(\bm x)=\bar{\bm{A}}_k \bm{x}+\bar{\bm a}_k\in \Re^{m_k}$ with $\bar{\bm{A}}_k\in \Re^{m_k\times n},\bar{\bm a}_k\in \Re^{m_k}$ such that $\bar{\bm{A}}_i$ and $\bar{\bm{A}}_j$, $\bar{\bm{a}}_i$ and $\bar{\bm{a}}_j$ do not overlap each $i\neq j$ with $\sum_{k\in[K]}m_k=m$.
\end{itemize}
\end{restatable}
\begin{proof}
We split the proof into two parts accordingly.
\begin{itemize}
\item[(i)] When $\bm a_k(\bm x)=\bm a_1(\bm x)$ for each $k\in[K]$, we can rewrite sDFO  \eqref{robust_max_eq} as 
\begin{align*}
v^*=\min_{\bm x\in\X} \max_{\bm\lambda\geq  \bm 0} \left\{\sum_{k\in[K]} \lambda_k b_k(\bm x)+ {\rxi^0}^\top \bm a_1(\bm x) - \theta \| \bm a_1(\bm x) \|_{1}\colon  \sum_{k\in[K]} \lambda_k =1  \right\}.
\end{align*}
Taking the dual of the inner maximization problem and using strong duality from linear programming, we have
\begin{align}
v^* = \min_{\bm x\in\X,\eta} \left\{\eta\colon \eta\geq {\rxi^0}^\top \bm a_1(\bm x) + b_k(\bm x)- \theta \| \bm a_1(\bm x) \|_{1}, \forall k\in[K] \right\}. \label{robust_max_micp_special_case_1}
\end{align}
According to Part (i) in Theorem~\ref{prop_min_micpr}, sDFO   \eqref{robust_max_micp_special_case_1} is MICP-R.

\item[(ii)] When $p=\infty$ and
$\rxi^\top \bm a_k(\bm x)= {\rxi_k}^\top \bar{\bm a}_{k}(\bm x) $ for each $k\in [K]$, we can rewrite sDFO \eqref{robust_max_eq} as 
\begin{align*}
  v^*=  \min_{\bm x\in\X}\max_{\bm\lambda\geq  \bm 0}\left\{\sum_{k\in[K]} \lambda_k \left[{\rxi_k^0}^\top \bar{\bm a}_{k}(\bm x) \right] + \sum_{k\in[K]} \lambda_k  b_k(\bm x) -  \theta \sum_{k\in[K]} \lambda_k\left\| \bar{\bm a}_{k}(\bm x) \right\|_{1}  \colon  \sum_{k\in[K]} \lambda_k =1 \right\},
\end{align*}
which can be simplified as 
\begin{align*}
v^*=   \min_{\bm x\in\X}\max_{\bm\lambda\geq  \bm 0}\left\{\sum_{k\in[K]} \lambda_k \left[{\rxi_k^0}^\top \bar{\bm a}_{k}(\bm x) \right] + \sum_{k\in[K]} \lambda_k b_k(\bm x) -  \theta  \sum_{k\in[K]} \lambda_k \left\| \bar{\bm a}_{k}(\bm x) \right\|_{1} \colon  \sum_{k\in[K]} \lambda_k =1 \right\}.
\end{align*}
Taking the dual of the inner maximization problem and using strong duality from linear programming, we have 
\begin{align}
v^*=   \min_{\bm x\in\X,\eta}\left\{\eta\colon \eta\geq  \left[{\rxi_k^0}^\top \bar{\bm a}_{k}(\bm x) \right] + b_k(\bm x) -  \theta \left\| \bar{\bm a}_{k}(\bm x) \right\|_{1}, \forall k\in[K] \right\}. \label{robust_max_micp_special_case_2}
\end{align}
According to Part (i) in  Theorem~\ref{prop_min_micpr}, sDFO  \eqref{robust_max_micp_special_case_2} is MICP-R.\QEDA
\end{itemize}
\end{proof}


The following Corollary~\ref{coro_max_micpr_good} shows the MICP-R formulations of the two cases discussed in Theorem~\ref{prop_max_micpr_good}.

\begin{corollary}
\label{coro_max_micpr_good}
When $p=\infty$, suppose that $\X\subseteq [\bm{l},\bm{u}]$.
\begin{itemize}
\item[(i)] If  $\bm a_k(\bm x)=\bm a_1(\bm x)$ for each $k\in[K]$, sDFO \eqref{robust_max_b} can be reformulated as the following MICP
\begin{align*}
v^* = \min_{\bm x\in\X,\eta} \quad & \eta, \\
\textup{s.t.} \quad & \eta\geq {\rxi^0}^\top \bm a_1(\bm x) + b_k(\bm x)-  \theta \sum_{i\in[m]}s_{1i}, \forall k\in[K],\\
& \left(s_{1i},z_{1i},a_{1i}(\bm x)\right) \in \mathcal{MI} \left(-1,1,\hat{l}_{1i},\hat{u}_{1i}\right),\forall i\in[m],
\end{align*}
where we let $\hat{l}_{ki}=\sum_{j\in [n]}\min\{\hat{A}_{kij} l_j,\hat{A}_{kij} u_j\}+\hat{a}_{ki}$ and $\hat{u}_{ki}=\sum_{j\in [n]}\max\{\hat{A}_{kij} l_j,\hat{A}_{kij} u_j\}+\hat{a}_{ki}$ for each $i\in [m]$ such that $\bm{a}_k(\bm x)\in [\hat{\bm{l}}_k,\hat{\bm{u}}_k]$ for each $k\in[K]$; and
\item[(ii)]  Suppose $\rxi:=[\rxi_1,\ldots,\rxi_K]$ such that $\rxi_i$ and $\rxi_j$ do not overlap for each $i\neq j$, and $\rxi^\top \bm a_k(\bm x)= \rxi_k^\top \bar{\bm a}_{k}(\bm x) $ for each $k\in[K]$, where $\rxi_k\in \Re^{m_k}$, $\bar{\bm a}_k(\bm x)=\bar{\bm{A}}_k \bm{x}+\bar{\bm a}_k\in \Re^{m_k}$ with $\bar{\bm{A}}_k\in \Re^{m_k\times n},\bar{\bm a}_k\in \Re^{m_k}$ such that $\bar{\bm{A}}_i$ and $\bar{\bm{A}}_j$, $\bar{\bm{a}}_i$ and $\bar{\bm{a}}_j$ do not overlap each $i\neq j$ with $\sum_{k\in[K]}m_k=m$ and each $m_k$ is nonnegative. Then, sDFO \eqref{robust_max_b} can be reformulated as the following MICP
\begin{align*}
v^* = \min_{\bm x\in\X,\eta} \quad & \eta,\\
\textup{s.t.} \quad 
& \eta\geq {\rxi^0}^\top \bm a_k(\bm x) + b_k(\bm x) - \theta \sum_{i\in[m]}s_{ki}, \forall k\in[K],\\
& \left(s_{ki},z_{ki},\bar{a}_{ki}(\bm x)\right) \in \mathcal{MI} \left(-1,1,\bar{l}_{ki},\bar{u}_{ki}\right),\forall k\in[K],i\in[m],
\end{align*}
where we let $\bar{l}_{ki}=\sum_{j\in [n]}\min\{\bar{A}_{kij} l_j,\bar{A}_{kij} u_j\}+\bar{a}_{ki}$ and $\bar{u}_{ki}=\sum_{j\in [n]}\max\{\bar{A}_{kij} l_j,\bar{A}_{kij} u_j\}+\bar{a}_{ki}$ for each $i\in [m]$ such that $\bar{\bm{a}}_k(\bm x)\in [\bar{\bm{l}}_k,\bar{\bm{u}}_k]$ for each $k\in[K]$. 
\end{itemize}
\end{corollary}
We remark that Corollary~\ref{coro_max_micpr_good} extends the implications outlined in Theorem~\ref{prop_max_micpr_good} by providing the MICP-R formulations of special cases of sDFO \eqref{robust_max_b} when $p=\infty$. Given that sDFO \eqref{eq_uq_simple} represents the special case of DFO \eqref{dfo}, the results derived in this section can directly contribute to the discussions of DFO \eqref{dfo} in the next section.

\section{DFO with the Ambiguity Set}
\label{sec_ambiguity_set}

Although DFO \eqref{dfo} is generally known to be NP-hard, in this section, we present sufficient conditions under which DFO \eqref{dfo} with different types of ambiguity sets can be tractable or MICP-R. We specifically investigate two representative ambiguity sets-- the type-$\infty$ Wasserstein ambiguity set and the ambiguity set with finite support.


\subsection{DFO \eqref{dfo} with Type-$\infty$ Wasserstein Ambiguity Set}
\label{DFO_Wasserstein_MICPR}
Our first goal is to expand upon the results in Section~\ref{sec_piecewise_affine}, particularly those related to tractability, complexity, and MICP-R formulations. We aim to apply and adapt these insights specifically to DFO \eqref{dfo}.
In particular, we focus on type-$\infty$ Wasserstein ambiguity set, which is defined as $\P_\infty^W =\{ \Pr\colon\Pr\{\trxi\in {\U}\}=1,W_\infty(\Pr,\Pr_{\trzeta})\leq \theta \}$, where 
$\Pr_{\trzeta}$ is a discrete empirical reference distribution of random parameters $\trzeta$ generated by $N$ i.i.d. samples such that $\Pr_{\trzeta}\{\trzeta=\rzeta^i\}=1/N$, i.e.,  $\Pr_{\trzeta} = 1/N\sum_{i\in[N]} \delta_{\rzeta^i}$ and $\delta_{\rzeta^i}$ is the Dirac function that places unit mass on the realization $\trzeta=\rzeta^i$ for each $i\in[N]$, and $\theta\geq0$ is the Wasserstein radius.
The type-$\infty$ Wasserstein distance between two probability distributions $\Pr_1,\Pr_2$ is defined as
\begin{equation*}
W_\infty(\Pr_1,\Pr_2)=\inf\left\{ {\rm{ess.sup}}_{\Qe}\left\|\rxi^1-\rxi^2\right\|_p\colon
\begin{aligned}
& \Qe \text{ is a joint distribution of } \trxi^1 \text{ and } \trxi^2\\
& \text{ with marginals }\Pr_1 \text{ and } \Pr_2, \text{ respectively }
\end{aligned}
\right\}.
\end{equation*}
In this setting, DFO \eqref{dfo} admits the following representation (see, e.g., \cite{bertsimas2023data,xie2020tractable}):
\begin{align}
v^*=\min_{\bm x\in\X} \left\{\inf_{\Pr\in\P_\infty^W}\E_\Pr\left[Q(\bm x, \trxi)\right] \right\}=\min_{\bm x\in\X}\left\{\frac{1}{N}\sum_{i\in[N]}\left[\inf_{\rxi}\left\{Q(\bm x, \rxi)\colon \|\rxi-\rzeta^i\|_p\leq \theta\right\}\right]\right\}. \label{infty_was}
\end{align}
Similar to the discussions in Section~\ref{eq_uq_simple_cases_concave} and Section~\ref{eq_uq_simple_cases_convex}, we then consider function $Q(\bm x,\rxi)$ to be convex and concave piecewise affine, respectively.

\noindent\textbf{Special Case I. Concave Piecewise Function.} We first consider the concave piecewise affine function. Following the same notation as Section~\ref{eq_uq_simple_cases_concave}, i.e., function $Q(\bm x, \rxi)$ is the minimum of piecewise affine functions, DFO \eqref{infty_was} can be written as
\begin{align}
v^*= \min_{\bm x\in\X}\frac{1}{N}\sum_{i\in[N]}\left[\min_{k\in[K]}{\rzeta^i}^\top \bm a_k(\bm x)+b_k(\bm x) - \theta\left\| \bm a_k(\bm x)\right\|_{p^*} \right]. \label{infty_was_concave} 
\end{align}
We notice that if there is only one sample available in the empirical distribution $\Pr_{\trzeta}$ (i.e., $N=1$), then DFO \eqref{infty_was_concave} reduces to sDFO \eqref{robust_min_b}. Thus, according to Theorem~\ref{prop_robust_min_tractable}, if $p\in(1,\infty]$, solving DFO \eqref{infty_was_concave} is, in general, NP-hard. It turns out that even with $p=1$, solving DFO \eqref{infty_was_concave} is also NP-hard.
\begin{restatable}{proposition}{propinftywasconcave}\label{prop_infty_was_concave} 
For any $p\in[1,\infty]$, solving DFO \eqref{infty_was_concave} is, in general, NP-hard.
\end{restatable}
\begin{proof}
For any $p\in(1,\infty]$, DFO \eqref{infty_was_concave} reduces to the favorable optimization \eqref{robust_max_b} if there is only one sample available for the empirical distribution $\Pr_{\trzeta}$ (i.e., $N=1$). Thus, according to Theorem~\ref{prop_robust_min_tractable}, solving DFO \eqref{infty_was_concave} is, in general, NP-hard. It remains to show that solving DFO \eqref{infty_was_concave} is also NP-hard when $p=1$.
Recall the NP-complete problem - the feasibility problem of a general binary program, which asks
\begin{quote}\it
\textbf{Feasibility of a binary program.} Given an integer matrix $\bm D \in \Ze^{\tau\times n}$, and integer vector $\bm d\in\Ze^\tau$, is there a vector $\bm x\in\{-1,1\}^n$ such that $\bm D\bm x\leq \bm d$? 
\end{quote}
Let us consider the following special case of DFO \eqref{infty_was_concave}. Let set $\X=\{\bm x: \bm D\bm x\leq \bm d, -1\leq x_i \leq 1,\forall i\in[n]\}$ and suppose $\theta=0$, $b_k(\bm x)=1$ for each $k\in[K]$, $\rzeta^i =\bm e_i$ for each $i\in[N]$, $N=n$, $K=2n,  m=n$, and 
\begin{align*}
\bm a_k(\bm x) = \left\{\begin{aligned}
x_k\bm e_k,
\quad &k\leq n, \\
-x_k\bm e_k, \quad  & n+1\leq k\leq 2n.
\end{aligned}
\right.
\end{align*}
Then, the inner minimum $\min_{k\in[K]}{\rzeta^i}^\top \bm a_k(\bm x)+b_k(\bm x)$  reduces to 
\begin{align*}
\min_{k\in[K]}{\rzeta^i}^\top \bm a_k(\bm x)+b_k(\bm x) = \min\left\{ 1,1-x_i,1+x_i \right\} = \min\left\{1,1-|x_i|\right\} =1-|x_i|.
\end{align*}
Thus, DFO \eqref{infty_was_concave} can be written as
\begin{align}
v^*= \min_{\begin{subarray}{c}
\bm D\bm x\leq \bm d, \\
\bm x\in[-1,1]^n
\end{subarray}}\frac{1}{n}\sum_{i\in[n]}\left[1-|x_i|\right]. \label{infty_was_concave_special_case}
\end{align}
We observe that the optimal value of DFO \eqref{infty_was_concave_special_case} $v^*=0$ if and only if $|x_i|=1$ for all $i\in [n]$, i.e., if and only if there exists a binary feasible solution $\bm x\in\{-1,1\}^n$ such that  $\bm D\bm x\leq \bm d$. Thus, solving problem \eqref{infty_was_concave_special_case} is NP-hard, and so is DFO \eqref{infty_was_concave}.
\QEDA
\end{proof}

Albeit being NP-hard, when $p\in\{1,\infty\}$, next theorem provides MICP-R formulation for  DFO \eqref{infty_was_concave}. 
\begin{restatable}{theorem}{propinftywasmicprconcave}\label{prop_infty_was_micpr_concave} 
Suppose $\X\subseteq [\bm{l},\bm{u}]$, let $\hat{l}^b_{k}=\sum_{j\in [n]}\min\{\hat{B}_{kj} l_j,\hat{B}_{kj} u_j\}+\hat{b}_{k}$ and $\hat{u}^b_{k}=\sum_{j\in [n]}\max\{\hat{B}_{kj} l_j,\hat{B}_{kj} u_j\}+\hat{b}_{k}$  such that $\bm{b}_k(\bm x)\in [\hat{{l}}^b_k,\hat{{u}}^b_k]$ for each $k\in[K]$, and
let $\hat{l}^a_{ki}=\sum_{j\in [n]}\min\{\hat{A}_{kij} l_j,\hat{A}_{kij} u_j\}+\hat{a}_{ki}$ and $\hat{u}^a_{ki}=\sum_{j\in [n]}\max\{\hat{A}_{kij} l_j,\hat{A}_{kij} u_j\}+\hat{a}_{ki}$ for each $i\in [m]$ such that $\bm{a}_k(\bm x)\in [\hat{\bm{l}}^a_k,\hat{\bm{u}}^a_k]$ for each $k\in[K]$. 
When $p\in\{1,\infty\}$, DFO \eqref{infty_was_concave} is MICP-R. 
\end{restatable}
\begin{proof}
We first introduce binary variables $\bm\lambda$ to reformulate the inner minimum, that is,
\begin{align*}
v^*= \min_{\bm x\in\X,\bm\lambda}\quad &\frac{1}{N}\sum_{i\in[N]}\sum_{k\in[K]}\lambda_{ki}\left[{\rzeta^i}^\top \bm a_k(\bm x)+b_k(\bm x) - \theta\left\| \bm a_k(\bm x)\right\|_{p^*}\right],\\
\textup{s.t.} \quad & \sum_{k\in[K]}\lambda_{ki}=1, \forall i\in[N],\\
& \lambda_{ki}\in\{0,1\},\forall k\in[K],i\in[N].
\end{align*}
Since set $\X$ is compact, we can apply McCormick inequalities \cite{mccormick1976computability} to linearize the terms \noindent $\{\lambda_{ki}\bm a_k(\bm x) \}_{i\in[N],k\in[K]}$ and $\{\lambda_{ki}b_k(\bm x)\}_{i\in[N],k\in[K]}$. It remains to provide the MICP-R formulation for the term $\{\lambda_{ki}\| \bm a_k(\bm x)\|_{p^*} \}_{i\in[N],k\in[K]}$. We split the discussions into two parts.
\begin{itemize}
\item[(i)] When $p=\infty$, i.e., the dual norm is $L_1$, the term $\{\lambda_{ki}\| \bm a_k(\bm x)\|_{1} \}$ can be linearized by applying McCormick inequalities twice for each $i\in[N]$ and $k\in[K]$. In this case,  DFO \eqref{infty_was_concave} is equivalent to the following MICP:
\begin{align*}
v^*= \min_{\bm x\in\X,\bm\lambda, \bm s^a,\bm s^b,\bm \eta}\quad &\frac{1}{N}\sum_{i\in[N]}\sum_{k\in[K]}  \sum_{j\in[m]}\zeta_{ij}s^a_{kij} +\frac{1}{N}\sum_{i\in[N]}\sum_{k\in[K]}s^b_{ki}-\frac{\theta }{N}\sum_{i\in[N]}\sum_{k\in[K]} \eta_{ki}, \\
\textup{s.t.} \quad & \sum_{k\in[K]}\lambda_{ki}=1, \forall i\in[N],\\
& \left(s^a_{kij},\lambda_{ki},a_{kj}(\bm x)\right)\in \mathcal{MI} (0,1,\hat{l}^a_{kj},\hat{u}^a_{kj}),\forall i\in[N],j\in[m], k\in[K],\\
& \left(s^b_{ki},\lambda_{ki},b_k(\bm x) \right)\in \mathcal{MI} (0,1,\hat{l}^b_{k},\hat{u}^b_{k}),\forall i\in[N], k\in[K],\\
&  \eta_{ki} \geq  s^a_{kij} , \forall i\in[N],j\in[m], k\in[K],\\
&  \eta_{ki} \geq - s^a_{kij} , \forall i\in[N],j\in[m], k\in[K].
\end{align*}

\item[(ii)] When $p=1$, i.e., the dual norm is $L_\infty$,
we can apply disjunctive programming \cite{balas1979disjunctive} to the terms $\{\lambda_{ki}\| \bm a_k(\bm x)\|_{\infty}\}$ and then apply McCormick inequalities to linearize 
$\{\lambda_{ki}\bm a_{kj}(\bm x) \}_{i\in[N],k\in[K],j\in[m]}$. In this case, DFO \eqref{infty_was_concave} is equivalent to the following MICP:
\begin{align*}
v^*= \min_{\bm x\in\X,\bm\lambda, \bm s^a,\bm s^b,\bm \eta}\quad &\frac{1}{N}\sum_{i\in[N]}\sum_{k\in[K]}  \sum_{j\in[m]}\zeta_{ij}s^a_{kij} +\frac{1}{N}\sum_{i\in[N]}\sum_{k\in[K]}s^b_{ki}-\frac{\theta }{N}\sum_{i\in[N]}\sum_{k\in[K]} \eta_{ki}, \\
\textup{s.t.} \quad & \sum_{k\in[K]}\lambda_{ki}=1, \forall i\in[N],\\
& \left(s^a_{kij},\lambda_{ki},a_{kj}(\bm x)\right)\in \mathcal{MI} (0,1,\hat{l}^a_{kj},\hat{u}^a_{kj}),\forall i\in[N],j\in[m], k\in[K],\\
& \left(s^b_{ki},\lambda_{ki},b_k(\bm x) \right)\in \mathcal{MI} (0,1,\hat{l}^b_{k},\hat{u}^b_{k}),\forall i\in[N], k\in[K],\\
&  \eta_{ki} \geq  s^a_{kij} , \forall i\in[N],j\in[m], k\in[K],\\
&  \eta_{ki} \geq - s^a_{kij} , \forall i\in[N],j\in[m], k\in[K].
\end{align*}
\end{itemize}
Therefore, according to the result in Lemma~\ref{micp_example}, DFO \eqref{infty_was_concave} is MICP-R with $p\in\{1,\infty\}$.
\QEDA
\end{proof}

Moreover, notice that sDFO \eqref{robust_min_b} is a special case of DFO \eqref{infty_was_concave} with $N=1$. According to Theorem~\ref{prop_min_micpr}, DFO \eqref{infty_was_concave} may not be MICP-R when $p\in(1,\infty)$.
\begin{corollary}
\label{corollary_infty_was_concave}
When $p\in(1,\infty)$, DFO \eqref{infty_was_concave} may not be MICP-R.
\end{corollary}
We remark that Corollary~\ref{corollary_infty_was_concave} shows the conditions under which DFO \eqref{infty_was_concave} may not be MICP-R.

\noindent\textbf{Special Case II. Convex Piecewise Function.}
Following the same notation introduced in Section~\ref{eq_uq_simple_cases_convex}, i.e., function $Q(\bm x, \rxi)$ is the maximum of piecewise affine function, DFO \eqref{infty_was} can be written as
\begin{align}
v^*= \min_{\bm x\in\X}\frac{1}{N}\sum_{i\in[N]}\left[\inf_{\rxi}\left\{\max_{k\in[K]}\rxi^\top \bm a_k(\bm x)+b_k(\bm x)\colon \|\rxi-\rzeta^i\|_p\leq \theta\right\}\right]. \label{infty_was_convex} 
\end{align}
In this special case, 
all the results in Section~\ref{eq_uq_simple_cases_convex} can be naturally extended to DFO \eqref{infty_was_convex}, as presented below. The proofs and formulations are omitted for brevity. 
\begin{corollary}
For DFO \eqref{infty_was_convex}, the complexity and tractability results in Section~\ref{eq_uq_simple_cases_convex} directly follow.
\end{corollary}

\subsection{DFO with the Ambiguity Set of Finite Support}
\label{sec_dfo_finite_support}

Many robust statistics recovered by DFO  \eqref{dfo} can be considered as DFO  \eqref{dfo} with the finite-support ambiguity set, i.e., when the support $\U:=\{\rxi^i\}_{i\in [N]}$ is finite (see more discussions in \cite{jiang2023dfo}). In this setting, we cannot obtain any nontrivial tractable results for DFO \eqref{dfo}. Furthermore, since evaluating the best case in DFO \eqref{dfo} with a given decision is, in general, NP-hard (see Proposition~\ref{prop_eva_dfo}), we instead focus on the MICP-R formulations of DFO \eqref{dfo} with finite support in this subsection.
Notably, when the ambiguity set $\P$ is a polytope (i.e., $\P=\left\{ \bm p\in\Re_+^{N}\colon \bm D\bm p \leq \bm d ,\e^\top \bm p =1 \right\}$ is a polytope with  $\bm D \in \Re^{\ell \times N}$ and $\bm d \in \Re^{\ell}$), we show that DFO \eqref{dfo} is MICP-R by observing that the number of the extreme points of the polyhedral ambiguity set $\P$ is finite. More specifically, DFO \eqref{dfo} with polyhedral ambiguity set $\P$ can be written as
\begin{equation}
v^*=\min_{\bm x \in \X}\min_{\bm p\geq \bm 0}\left\{\sum_{i\in [N]}p_iQ(\bm x,\rxi^i)\colon \bm D\bm p \leq \bm d ,\e^\top \bm p =1 \right\}.\label{dfo_finite}
\end{equation}

\begin{restatable}{theorem}{polymixedrepresentable}\label{poly_mixed_representable} 
Suppose that both set $\X$ and function $Q(\bm x,\rxi)$ are MICP-R. Then, the corresponding DFO \eqref{dfo_finite} under a polyhedral ambiguity set $\P=\left\{ \bm p\in\Re_+^{N}\colon \bm D\bm p \leq \bm d ,\e^\top \bm p =1 \right\}$, is MICP-R. 
\end{restatable}
\begin{proof}
Since the polyhedral ambiguity set $\P$ is a polytope, we can enumerate all its extreme points, i.e., $\bm\gamma^1,\dots, \bm\gamma^s \in \Re_+^N$ are the total $s$ vertices of $\P$. Then, DFO \eqref{dfo_finite} is equivalent to
\begin{align*}
v^*=\min_{j\in [s]}\min_{\bm x\in\X}\sum_{i\in [N]}\gamma^j_iQ(\bm x,\rxi^i),
\end{align*}
which is MICP-R, since both set $\X$ and function $Q(\bm x,\rxi)$ are MICP-R.
\QEDA
\end{proof}

We remark that the proof of Theorem~\ref{poly_mixed_representable} relies on the enumeration of extreme points of the polyhedral ambiguity set $\P$, which can be computationally inefficient. Instead of enumerating all the extreme points of the polyhedral ambiguity set $\P$, there exists an alternative MICP-R formulation of the corresponding DFO \eqref{dfo_finite}, which is based on the KKT condition.

\begin{restatable}{theorem}{polymixedrepresentableequivalentinteger}\label{poly_mixed_representable_equivalent_integer} 
Suppose that $\bm D \in \Ze^{\ell \times N}$ in the polyhedral ambiguity set $\P$ and $N \geq \ell+1$, there exists an $\bar{L}$ such that the largest row encoding length of the matrix $\begin{bmatrix} \bm d^\top & 1 \\ \bm D^\top & \bm e \end{bmatrix}$ is $\bar{L}$.
Then, under the same presumptions in Theorem~\ref{poly_mixed_representable}, 
the corresponding DFO \eqref{dfo_finite} can be written as:
\begin{align}
v^*=\min_{\bm x,\bm v,t,\bm\alpha,\beta,\bm z, \bar{\bm z}, \bm p}\left\{t\colon
\begin{array}{cc}
\displaystyle   \bm\alpha^\top \bm d +\beta \leq t, \bm D^\top\bm\alpha+\beta\e \leq \bm v,\\
\displaystyle  0\leq \alpha_j \leq M_{j,1} z_j,\forall j\in[\ell], 0\leq  d_j - \bm D_{j\cdot}\bm p \leq M_{j,2} (1-z_j), \forall j\in[\ell],\\
\displaystyle 0\leq p_j \leq \bar{z}_j, \forall i\in[N], \bm e^\top \bm p =1, \\
\displaystyle 0\leq  v_i - \bm D_{\cdot i}^\top \bm \alpha -\beta \leq \bar{M}_i (1-\bar{z}_i),\forall i\in[N],\\
\displaystyle v_i\geq Q(\bm x,\rxi^i), \forall i\in[N],  \bm z\in\{0,1\}^{\ell}, \bar{\bm z}\in\{0,1\}^{N}, \bm x\in\X
\end{array} 
\right\}, \label{dfo_poly_mixed_representable_equivalent}
\end{align}
where $M_{j,1}$ and $M_{j,2}$ are valid upper bounds for  $\alpha_j$ and $ \bm d_j - \bm D_{j\cdot}\bm p$, respectively, for each $j\in[\ell]$, and $\bar{M}_i$ is the valid upper bound for $v_i - \bm D_{\cdot i}^\top \bm \alpha -\beta$ for each $i\in[N]$, i.e., letting 
${U}_{i,1}=\max_{\bm x\in\X} Q(\bm x,\rxi^i)$ for each $i\in[N]$, and the corresponding big-M coefficients in DFO \eqref{dfo_poly_mixed_representable_equivalent} can be found as:
\begin{align*}
& M_{j,1} \geq {U}_{i,1}(\ell+1)2^{\bar{L}}, \forall j\in[\ell],M_{j,2} \geq  d_j + \sum_{i\in[N]}|D_{ji}|, \forall j\in[\ell], \\
& \bar{M}_i \geq {U}_{i,1} \left(\sum_{j\in[\ell]}|D_{ji}|+1\right)(\ell+1)2^{\bar{L}} + \min_{i\in[N]}{U}_{i,1}(\ell+1)2^{\bar{L}}, \forall i\in[N].
\end{align*}
\end{restatable}

\begin{proof}
We split the proof into four steps.

\noindent{\textbf{Step I.}} 
Introducing the slack variables $\bm v$ and $t$ for DFO \eqref{dfo_finite}, we have
\begin{align*}
v^*=\min_{\bm x\in\X}\min_{\bm v,t}\left\{t\colon  \min_{\bm p\geq \bm 0}\left\{\sum_{i\in [N]}p_i v_i\colon \bm D\bm p \leq \bm d, \e^\top \bm p =1 \right\} \leq t, v_i\geq Q(\bm x,\rxi^i), \forall i\in[N] \right\}.
\end{align*}
For the innermost minimization of the problem over $\bm p$, we can use the complementary slackness to obtain the equivalent reformulation with big-M coefficients, that is,
\begin{align*}
v^*=\min_{\bm x\in\X,\bm v,t,\bm\alpha,\beta,\bm z,\bm p}\left\{t\colon
\begin{array}{cc}
\displaystyle  \bm\alpha^\top \bm d +\beta \leq t, \bm D^\top\bm\alpha+\beta\e \leq \bm v,\\
\displaystyle  0\leq \alpha_j \leq M_{j,1} z_j,\forall j\in[\ell], 0\leq  d_j - \bm D_{j\cdot}\bm p \leq M_{j,2} (1-z_j), \forall j\in[\ell],\\
\displaystyle 0\leq p_j \leq \bar{z}_j, \forall i\in[N], \bm e^\top \bm p =1, \\
\displaystyle 0\leq  v_i - \bm D_{\cdot i}^\top \bm \alpha -\beta \leq \bar{M}_i (1-\bar{z}_i),\forall i\in[N],\\
\displaystyle v_i\geq Q(\bm x,\rxi^i), \forall i\in[N],  \bm z\in\{0,1\}^{\ell}, \bar{\bm z}\in\{0,1\}^{N}
\end{array} 
\right\}.
\end{align*}

\noindent{\textbf{Step II.}} 
We then demonstrate that
set 
\begin{align*}
 \bar{\Theta} = \left\{ (\bm\alpha,\beta) \colon  \begin{bmatrix} 
    \bm d^\top & 1 \\ 
    \bm D^\top & \bm e \\
    - \bm I & \bm 0
    \end{bmatrix}
     \begin{bmatrix} 
    \bm \alpha \\ 
    \beta
    \end{bmatrix} 
    \leq 
      \begin{bmatrix} 
    t\\ 
    \bm v\\
    \bm 0
    \end{bmatrix} \right\}
\end{align*}
is nonempty and contains no line. 
It is evident that the point $(\bm\alpha=\bm 0,\beta=\min_{i\in[N]}v_i \in \bar{\Theta}$, indicating the existence of at least one point in set $\bar{\Theta}$. Hence, set $\bar{\Theta}$ is nonempty. 
Suppose there exists a vector $\bm d$ such that for every point $(\bar{\bm \alpha},\bar{\beta})$ within set $\bar{\Theta}$, $(\bar{\bm \alpha},\bar{\beta})+\lambda \bm d$ is also in set $\bar{\Theta}$ for all $\lambda\in\Re$. That is,
\begin{align*}
    \begin{bmatrix} 
    \bm d^\top & 1 \\ 
    \bm D^\top & \bm e \\
    - \bm I & \bm 0
    \end{bmatrix}
    \left( \begin{bmatrix} 
    \bar{\bm \alpha} \\ 
    \bar{\beta}
    \end{bmatrix} +\lambda \bm d\right)
    \leq 
      \begin{bmatrix} 
    t\\ 
    \bm v\\
    \bm 0
    \end{bmatrix}, \forall \lambda\in\Re. 
\end{align*}
Therefore, we have $\bm d= \bm 0$.
Hence, set $\bar{\Theta}$ contains no line, which implies that set $\bar{\Theta}$ contains an extreme point.

\noindent{\textbf{Step III.}} 
We now proceed to demonstrate that each extreme point in $\bar{\Theta}$ is bounded. Since there exists an $\bar{L}$ such that the row encoding length of every subset of $(\ell+1)$ rows from the matrix $\begin{bmatrix} \bm d^\top & 1 \\ \bm D^\top & \bm e \end{bmatrix}$ is $\bar{L}$ (see the details in \cite{grotschel2012geometric}), we have 
\begin{align*}
\mathrm{ext}\left\{ (\bm\alpha,\beta) \colon  \begin{bmatrix} 
    \bm d^\top & 1 \\ 
    \bm D^\top & \bm e \\
    - \bm I & \bm 0
    \end{bmatrix}
     \begin{bmatrix} 
    \bm \alpha \\ 
    \beta
    \end{bmatrix} 
    \leq 
      \begin{bmatrix} 
    t\\ 
    \bm v\\
    \bm 0
    \end{bmatrix} \right\}  
    \subseteq
    \left\{ (\bm\alpha,\beta) \colon  
    \alpha_j \leq {U}_{i,1}(\ell+1)2^{\bar{L}}, \forall j\in[\ell], |\beta| \leq  \min_{i\in[N]}{U}_{i,1}(\ell+1)2^{\bar{L}} 
    \right\}.
\end{align*}

\noindent{\textbf{Step IV.}} 
Next, we determine the values of big-M coefficients. By substituting the upper bounds of $\bm\alpha$ and $\beta$, we can determine the explicit values for the big-M coefficients in DFO \eqref{dfo_poly_mixed_representable_equivalent}. Specifically, the coefficients can be found as follows:
\begin{align*}
  & M_{j,1} \geq {U}_{i,1}(\ell+1)2^{\bar{L}}\geq \alpha_j,\forall j\in[\ell],\\
  &  M_{j,2} \geq  d_j + \sum_{i\in[N]}|D_{ji}|\geq  d_j  - \bm D_{j\cdot}\bm p, \forall j\in[\ell],\\
  & \bar{M}_i  \geq {U}_{i,1} \left(\sum_{j\in[\ell]}|D_{ji}|+1\right)(\ell+1)2^{\bar{L}} + \min_{i\in[N]}{U}_{i,1}(\ell+1)2^{\bar{L}}   \geq   v_i - \bm {D}_{\cdot i}^\top \bm \alpha -\beta, \forall i\in[N].
\end{align*}
This completes the proof.
\QEDA
\end{proof}

We remark that based on the results in Theorem~\ref{poly_mixed_representable_equivalent_integer}, following the same presumptions in Theorem~\ref{poly_mixed_representable} with nonnegative matrix $\bm D$ and positive vector $\bm d$, we can obtain similar results. Below is an example.

\begin{restatable}{corollary}{polymixedrepresentableequivalent}\label{poly_mixed_representable_equivalent} 
Under the same presumptions as that in Theorem~\ref{poly_mixed_representable} and the assumptions that $\bm D\geq \bm 0$ and $\bm d >\bm 0$ in the polyhedral ambiguity set $\P$, the corresponding big-M coefficients in DFO \eqref{dfo_poly_mixed_representable_equivalent} can be found as:
\begin{align*}
& M_{j,1}  \geq  \frac{1}{d_j}\left[ \max_{i\in[N]}{U}_{i,1} - \min_{i\in[N]}{L}_{i,1} \right], \forall j\in[\ell],\\
& M_{j,2} \geq  d_j, \forall j\in[\ell], \bar{M}_i \geq {U}_{i,1}-\min_{i\in[N]}{L}_{i,1}, \forall i\in[N],
\end{align*}
where $M_{j,1}$ and $M_{j,2}$ are valid upper bounds for  $\alpha_j$ and $ \bm d_j - \bm D_{j\cdot}\bm p$, respectively, for each $j\in[\ell]$, and $\bar{M}_i$ is the valid upper bound for $v_i - \bm D_{\cdot i}^\top \bm \alpha -\beta$ for each $i\in[N]$, i.e., letting 
${L}_{i,1}=\min_{\bm x\in\X} Q(\bm x,\rxi^i)$ for each $i\in[N]$, ${U}_{i,1}=\max_{\bm x\in\X} Q(\bm x,\rxi^i)$ for each $i\in[N]$.
\end{restatable}

\begin{proof}
Under the conditions that $\bm D\geq \bm 0$, $\bm d >\bm 0$, and set $\X$ is compact, we know $t$ in DFO \eqref{dfo_poly_mixed_representable_equivalent} is bounded by $t\in [\min_{i\in[N]}{L}_{i,1}, \max_{i\in[N]}{U}_{i,1}]$, and $\beta$ in DFO \eqref{dfo_poly_mixed_representable_equivalent} is bounded by $\beta\in [\min_{i\in[N]}{L}_{i,1}, \max_{i\in[N]}{U}_{i,1}]$. Then, we obtain explicit values for the valid upper bounds $\{M_{j,1}\}_{j\in[\ell]}$, $\{M_{j,2}\}_{j\in[\ell]}$, $\{\bar{M}_i\}_{i\in[N]}$. This completes the proof.
\QEDA
\end{proof}


\noindent\textbf{A Special Case: MICP-R Formulation for Interval Polyhedral Ambiguity Set.} Building on Theorem~\ref{poly_mixed_representable} and Theorem~\ref{poly_mixed_representable_equivalent_integer}, 
we provide a compact MICP-R formulation for a particular type of ambiguity set, namely the interval polyhedral ambiguity set, which is defined as $\P_{I}=\{ \bm p= \bm p^0 + \bm \psi \in\Re_+^N \colon \bm l \leq\bm \psi \leq  \bm u, \bm e^\top \bm \psi=0 \}$.
Here, we let $\bm p^0$ denote the nominal probability vector with $\bm p^0 \geq \bm 0 $ and $\sum_{i\in[N]}p_i^0=1$, the lower bound vector $\bm{l}\geq-\bm p^0$ 
and the bounds $l_i=\bar l_i/q, u_i=\bar u_i/q$ with $q$ being a positive integer and $\bar l_i,\bar u_i$ being integers for each $i\in[N]$.

\begin{restatable}{corollary}{polymixedrepresentableupperlower}\label{poly_mixed_representable_upper_lower} 
Suppose both set $\X$ and function $Q(\bm x,\rxi)$ are MICP-R. Then under the interval polyhedral ambiguity set $\P_{I}$, the optimal value of the corresponding DFO \eqref{dfo} is $v^*=\min_{j\in [N],\tau\in [\bar{l}_j,\bar{u}_j]}v_{j\tau}^*$ and for each $j\in [N]$ and $\tau \in\{\bar{l}_j,\bar{l}_j+1,\cdots,\bar{u}_j\}$,
the value $v_{j\tau}^*$ can be computed by solving the following MICP-R formulation:
\begin{align*}
v_{j\tau}^*= \min_{\begin{subarray}{c}
\bm x\in\X,\bm\eta,\bm \nu,\\ \bm z^j\in\{0,1\}^N
\end{subarray} } & \, \sum_{i\in[N]\setminus \{j\}}\left[p^0_i+\bar l_i/q\right]\nu_i+ \sum_{i\in[N]\setminus \{j\}}(\bar u_i/q-\bar l_i/q)\eta^j_{i}+ \left(p^0_j+\tau/q\right) \nu_j,\\
\textup{s.t.}\quad & \nu_i\geq Q(\bm x,\rxi^i), \left(\eta^j_{i},z^j_{i},\nu_i\right) \in \mathcal{MI} \left(0,1,L_i,U_i\right),\forall i\in[N], j\in[N],\\
& -\sum_{i\in[N]\setminus \{j\}} (\bar l_i+\left(\bar u_i-\bar l_i)z^j_i\right) =\tau,
\end{align*}
where for each $i\in[N]$, $L_i$ and $U_i$ are the valid lower and upper bounds of the function $Q(\bm x,\rxi^i)$, respectively.
\end{restatable}
\begin{proof}
DFO \eqref{dfo} is equivalent to
\begin{align}
v^*=\min_{\bm x\in\X}\min_{\bm \psi}\left\{\sum_{i\in[N]}\left[ (p^0_i+\psi_i)Q(\bm x,\rxi^i) \colon  \sum_{i\in[N]}\psi_i=0, \bar l_i/q\leq \psi_i \leq \bar u_i/q,\forall i\in[N]\right]\right\}. \label{poly_mixed_representable_simplify}
\end{align}
According to the extreme point characterization of the ambiguity set $\P_{I}$, for any extreme point $\hat{\bm{\psi}}$, it has at least $N-1$ components taking values from $\bm l$ or $\bm{u}$ and one component corresponding to equality constraint. Let us assume that component $j\in [N]$ corresponds to the equality constraint. Accordingly, we can define the binary variable $z^j_i\in\{0,1\}$ for each $i\in[N]$ and $\hat{\psi}_i= l_i+(u_i-l_i)z_i^j$ for each $i\in[N]\setminus \{j\}$. Since we have $\sum_{i\in [N]}\hat{\psi}_i=0$, thus
$\hat \psi_j=-\sum_{i\in[N]\setminus\{j\}} (l_i+(u_i-l_i)z^j_i)$. Plugging the extreme point representation into set $\P_{I}$, we have
\begin{align*}
\P_{I}=\conv\left[\bigcup_{j\in [N]}\left\{\bm p= \bm p^0 + \hat{\bm \psi} \in\Re_+^N \colon \begin{array}{cc}
\displaystyle\hat{\psi}_i= l_i+(u_i-l_i)z_i^j,  \forall i\in[N]\setminus \{j\}, \\
\displaystyle \hat \psi_j=-\sum_{i\in[N],i\setminus\{j\}} (l_i+(u_i-l_i)z^j_i) \in [l_j,u_j]
\end{array} \right\}\right].
\end{align*}
Plugging the representation of set $\P_I$, DFO \eqref{poly_mixed_representable_simplify} is equivalent to $v^*=\min_{j\in [N]}v_j^*$ and for each $j\in [N]$,
\begin{align*}
v_j^*=\min_{\begin{subarray}{c}
\bm x\in\X,\\ \bm z^j\in\{0,1\}^N
\end{subarray} } &  \sum_{i\in[N]\setminus\{j\} }\left[p^0_i+l_i+(u_i-l_i)z^j_i \right]Q(\bm x,\rxi^i) + \left(p^0_j-\sum_{i\in[N],i\setminus\{j\}} \left(l_i+(u_i-l_i)z^j_i\right)\right) Q(\bm x,\rxi^j),\\
\text{s.t.}\quad&l_j\leq -\sum_{i\in[N]\setminus\{j\}}\left (l_i+(u_i-l_i)z^j_i\right) \leq u_j.
\end{align*}
Since $l_i=\bar l_i/q, u_i=\bar u_i/q$ for each $i\in[N]$, then for each $j\in[N]$, 
the expression $-\sum_{i\in[N]\setminus\{j\}} (l_i+(u_i-l_i)z^j_i)$ can take values from $\{\tau/q\}_{\tau\in [\bar{l}_j,\bar{u}_j]}$ and $\tau$ is an integer. This fact allows us to simplify $v_j^*=\min_{\tau\in [\bar{l}_j,\bar{u}_j]}v_{j\tau}^*$, where
\begin{align*}
v_{j\tau}^*=\min_{\begin{subarray}{c}\bm x\in\X, \bm\nu,\\ {\bm z}^j\in\{0,1\}^N\end{subarray} } &\sum_{i\in[N]\setminus\{j\}}\left[p^0_i+\frac{\bar l_i}{q}+\frac{1}{q}(\bar u_i-\bar l_i)z^j_i \right]\nu_i + \left(p^0_j+\frac{\tau}{q}\right) \nu_j,\\
\text{s.t.}\quad& -\sum_{i\in[N]\setminus\{j\}} \left(\bar l_i+(\bar u_i-\bar l_i)z^j_i\right) =\tau, \nu_i\geq Q(\bm x,\rxi^i), \forall i\in [N],
\end{align*}
for each $j\in [N]$ and $\tau\in [\bar{l}_j,\bar{u}_j]$.
Since set $\X$ is compact, we can apply the McCormick inequalities \citep{mccormick1976computability} to linearize the bilinear terms $\{z^j_i\nu_i\}_{i\in[N],j\in[N]}$, this completes the proof.
\QEDA
\end{proof}

We remark that as a direct application of Corollary~\ref{poly_mixed_representable_upper_lower}, when $l_i=l,u_i=u$ for each $i\in[N]$, the MICP-R formulation of Corollary~\ref{poly_mixed_representable_upper_lower} can be further simplified.

\begin{corollary}
\label{mixed_representable_upper_lower_fcvar}
Suppose that the premises of Corollary~\ref{poly_mixed_representable_upper_lower} hold and $p^0_i=1/N, l_i=-1/N,u_i=u$ for each $i\in[N]$.
Then the optimal value of the corresponding DFO \eqref{dfo} is $v^*=\min_{j\in [N]}v_{j}^*$ and for each $j\in [N]$, the value $v_{j}^*$ can be computed via the following MICP-R formulation:
\begin{align*}
v_{j}^*= \min_{\begin{subarray}{c}
\bm x\in\X,\bm\eta,\bm \nu,\\ \bm z^j\in\{0,1\}^N
\end{subarray} } & \, \sum_{i\in[N] \setminus \{j\}}(u+1/N)\nu_i+(1-\floor{\kappa}/\kappa ) \nu_j,\\
\textup{s.t.}\quad & \nu_i\geq Q(\bm x,\rxi^i), \left(\eta^j_{i},z^j_{i},\nu_i\right) \in \mathcal{MI} \left(0,1,L_i,U_i\right),\forall i\in[N], j\in[N],\\
& \sum_{i\in[N]\setminus \{j\}} z^j_i =\lfloor \kappa \rfloor,
\end{align*}
where $\kappa=N/(uN+1)$ and for each $i\in[N]$, $L_i$ and $U_i$ are the lower and upper bounds of the function $Q(\bm x,\rxi^i)$, respectively.
\end{corollary}
The result in Corollary~\ref{mixed_representable_upper_lower_fcvar} will be demonstrated in the numerical study section. It is important to note, however, that the MICP-R result in Theorem~\ref{poly_mixed_representable} does not hold when the ambiguity set with finite support is not polyhedral. 

\begin{restatable}{proposition}{notpolymixedrepresentable}\label{not_poly_mixed_representable} 
Suppose that the ambiguity set is  $\P=\{\bm p \colon \| \bm p- \bm p^0\|_2\leq \theta,\sum_{i\in[N]}p_i=1, \bm p\geq \bm 0 \}$,
where $\bm p^0 =\bm e/N$ denotes the nominal probability. When $0<\theta\leq \sqrt{1/(N(N-1))}$, DFO \eqref{dfo} may not be MICP-R. 
\end{restatable}
\begin{subequations}
\begin{proof}
Let us consider a simple function $Q(\bm x,\rxi^i)=x_i$ for each $i\in[N]$ and $\bm c=\bm 0$. Then, DFO \eqref{dfo} is equivalent to 
\begin{align}
v^*=\min_{\bm x\in \X}\min_{\bm p\geq \bm 0}\left[ \sum_{i\in[N]}p_ix_i\colon  \sum_{i\in[N]}p_i=1,  \left\| \bm p-\frac{1}{N}\bm e\right\|_2\leq \theta\right]. \label{proof_finite_support_not_micpr_a}
\end{align}
Let us focus on simplifying the inner minimization of DFO \eqref{proof_finite_support_not_micpr_a} and define $\bm y =\bm x-(\bm x ^\top \bm e)\bm e/N $. Then, by the definition, we must have $\sum_{i\in[N]}y_i=0$ and
\begin{align*}
\sum_{i\in[N]}p_iy_i  =\sum_{i\in[N]}p_ix_i - \frac{1}{N}\sum_{i\in[N]}(x_i-1).
\end{align*}
DFO \eqref{proof_finite_support_not_micpr_a} is equivalent to
\begin{align}
v^*=\min_{\begin{subarray}{c}
\bm x\in \X,\\
\bm y =\bm x-(\bm x ^\top \bm e)\bm e/N
\end{subarray}}\frac{1}{N}\sum_{i\in[N]}(x_i-1)+\min_{\bm p\geq \bm 0}\left\{ \sum_{i\in[N]}p_iy_i \colon \sum_{i\in[N]}p_i=1,  \left\| \bm p-\frac{1}{N}\bm e\right\|_2\leq \theta \right\}.\label{proof_finite_support_not_micpr_b}
\end{align}
Letting $\hat {\bm p} = \bm p-\bm e/N$, DFO \eqref{proof_finite_support_not_micpr_b} is simplified as
\begin{align}
v^*=\min_{\bm y}\min_{\hat {\bm p} \geq -\bm e/N }\left\{ \sum_{i\in[N]}\hat{p}_iy_i \colon \sum_{i\in[N]}\hat{p}_i=0,  \left\| \hat{\bm p}\right\|_2\leq \theta, \sum_{i\in[N]}y_i=0 \right\}. \label{proof_finite_support_not_micpr_b_1}
\end{align}
According to the H\"older's inequality, the inner minimization of DFO \eqref{proof_finite_support_not_micpr_b_1} can be lower bounded by
\begin{align*}
\min_{\hat {\bm p} \geq -\bm e/N} \left\{\sum_{i\in[N]}\hat{p}_iy_i \colon \sum_{i\in[N]}\hat{p}_i=0,  \left\| \hat{\bm p}\right\|_2\leq \theta\right\} \geq  -\theta \|\bm y\|_2.
\end{align*}
In fact, the above equality can be achieved by the solution $\hat{p}_i^*=-\theta y_i/\|\bm y\|_2 $ for all $i\in [N]$. Since $\sum_{i\in [N]}\hat{p}_i^*=0$ and $\|\hat{\bm p}^*\|_2=\theta$, it suffices to show that 
\begin{align}
-|\hat{p}_i^*|=-\theta|y_i|/\|\bm y\|_2 \geq -\frac{1}{N},\label{proof_finite_support_not_micpr_c}
\end{align}
for all $\sum_{i\in [N]}\hat{p}_i^*=0$. That is, we need to show that
\begin{align}
\max_{\sum_{i\in[N]}y_i=0} \frac{|y_i|}{\|\bm y\|_2} \leq \frac{1}{N\theta}. \label{proof_finite_support_not_micpr_d}
\end{align}
Without loss of generality, suppose that $y_{\ell}\neq 0$.
Letting $y'_{i}=y_i/y_{\ell}$ for each $i\in[N]\setminus \{\ell\}$, the condition $\sum_{i\in[N]}y_i=0$ is equivalent to  $\sum_{i\in[N]}y'_{i}=-1$. Thus,
\begin{align*}
\max_{\sum_{i\in[N]}y_i=0} \frac{|y_i|}{\|\bm y\|_2} =\min_{\bm y'}\left\{\sqrt{1+\sum_{i\in[N]\setminus\ell} {y'_i}^2}\colon \sum_{i\in[N]}{y'_i}=-1\right\} = \sqrt{1+\left(\frac{1}{\sqrt{N-1}}\right)^2} = \sqrt{\frac{N}{N-1}}.
\end{align*}
Hence, the inequality in \eqref{proof_finite_support_not_micpr_d} must be satisfied since 
\begin{align*}
\max_{\sum_{i\in[N]}y'_{i}=-1}\, \frac{1}{\sqrt{1+\sum_{i\in[N]\setminus\ell} {y'_i}^2}} = \sqrt{\frac{N-1}{N}} \leq \frac{1}{N\theta}
\end{align*}
and $0<\theta\leq \sqrt{1/(N(N-1))}$.
Therefore, plugging in $\bm y =\bm x-(\bm x ^\top \bm e)\bm e/N$, DFO \eqref{proof_finite_support_not_micpr_b_1} is equivalent to 
\[v^*=\min_{\begin{subarray}{c}
\bm x\in \X
\end{subarray}}\frac{1}{N}\sum_{i\in[N]}(x_i-1)-\frac{\theta }{N}\left\|N\bm x-(\bm x ^\top \bm e)\bm e\right\|_2.\]
Using the fact that set $\X$ has a nonempty relative interior and following the similar proof as that of Lemma~\ref{micp_example}, we conclude that this DFO cannot be MICP-R. 
\QEDA
\end{proof}
\end{subequations}

The findings in Proposition~\ref{not_poly_mixed_representable} reveal that even when DFO \eqref{dfo} is based on a finite-support ambiguity set, it may not always be MICP-R. This highlights the necessity of carefully selecting an appropriate ambiguity set for DFO problems, especially in the context of data-driven decision-making. Additionally, it is important to note that the insights gained in this subsection are not limited to static uncertainty; they can be extended to encompass decision-dependent uncertainty. This includes applications in two-stage stochastic programs with decision-dependent uncertainty, as explored in recent works such as \cite{zhang2020unified, zeng2022two, vayanos2020robust, luo2020distributionally}.

\section{Numerical Study}
\label{sec_numerical}
To demonstrate the value of the MICP-R formulations, we consider the interval polyhedral ambiguity set $\P_{I}$ and apply the result in Corollary~\ref{mixed_representable_upper_lower_fcvar}. All instances in this section are coded in Python 3.9 with calls to solver Gurobi (version 9.5.2 with default settings) on a personal PC with an Apple M1 Pro processor and 16GB of memory. We set the time limit of each instance to be $3600$s. 

In particular, we consider a two-stage resource allocation (TRA) problem (also studied by \cite{duque2022distributionally}), which consists of a set of facilities, denoted by $s\in[n]$, that can be used to meet the demand from the customer sites, denoted by $j\in[n_1]$. In the TRA problem, the first-stage decision is to distribute a single type of commodity across these facilities. Once the allocation for each facility is determined, we then meet the demand in the second stage at the lowest possible cost. When the supply is insufficient, a large unit penalty (i.e., outsourcing) cost $\rho$ will be incurred for unsatisfied demand. Conversely, surplus supply at any facility has to bear a unit holding cost, $h$. The TRA problem can be formulated as
\begin{subequations}\label{eq_tcfl}
\begin{align}
\min _{\bm x\geq \bm 0} \E_{\Pr}\left[Q(\bm{x},{\trxi})\right],\label{second_stage_trp_obj}
\end{align}
where for a realization ${\rxi}$,
\begin{align}
Q(\bm x, {\rxi})=\min _{\bm y\geq \bm 0,\bm u\geq \bm 0, \bm v\geq \bm 0}  \left\{ \sum_{s\in[n]} \sum_{j\in[n_1]} q_{sj}y_{sj}+\rho \sum_{j\in[n_1]} u_j + h \sum_{s\in[n]} v_s \colon
\begin{array}{l}
\displaystyle
\sum_{j\in[n_1]} y_{sj} +v_s=x_s, \forall s\in[n], \\
\displaystyle \sum_{s\in[n]} y_{sj}+u_j \geq  \xi_j,  \forall j\in[n_1] 
\end{array}\right\}.\label{second_stage_trp_const}
\end{align}
\end{subequations}
In the TRA \eqref{eq_tcfl}, for each $s\in [n]$, the variable $x_s$ denotes the supply allocated to facility $s$. For $s\in[n]$ and $j\in [n_1] $, variable $y_{sj}$ represents the amount of the demand from the customer site $j$ satisfied by facility $s$, with the associated transportation cost denoted by $q_{sj}$. The parameters $\trxi$ are random, where $\tilde\xi_j$ denotes the random demand of customer site $j\in[n_1]$.

In the numerical experiments, we solve TRA \eqref{eq_tcfl} under finite support by generating random instances with varying sample sizes $N$. All the random variables (i.e., the customer demands $\trxi$) are truncated to be nonnegative. For each instance, we assume that the transportation cost vector $\bm{q}$ components are i.i.d.~truncated Gaussian with mean $1$ and variance $0.2$. The components of the customer demand $\trxi$ are i.i.d.~truncated Gaussian random variables with means $\bar{d}/n_1$ and variances, $0.005\times\bar{d}$ with $\bar{d}=1000$.
We also assume some outliers exist in the customer demand information, denoted by $\tilde{\bm{\xi}}^o$.  
We assume the components of random vector $\tilde{\bm{\xi}}^o$ are i.i.d.~truncated Gaussian with mean $\bar{d}/n_1$ and variance $0.01\times\bar{d}$.
The observed demand vector follows the following distribution {$0.95\tilde{\bm{\xi}}+0.05\tilde{\bm{\xi}}^o$. We set the number of potential facilities $n=8$, the number of customers $n_1=20$, the unit penalty cost $\rho=10$, and the unit cost for holding inventory $h=1$.

\noindent\textbf{Experiment 1. Value of MICP-R from Corollary~\ref{mixed_representable_upper_lower_fcvar}.} In the numerical implementation, we use DFO \eqref{dfo} to reduce the effect of outliers in the original SAA problem \eqref{second_stage_trp_obj}. Particularly, we consider the formulation in Corollary~\ref{mixed_representable_upper_lower_fcvar}, that is,
\begin{align}
\min _{\bm x\geq \bm 0} \inf_{\Pr\in\P_{I}}\E_{\Pr}\left[Q(\bm x, \trxi)\right],\label{second_stage_trp_obj_dfo}
\end{align}
We set $\mu=1/(N-N\varepsilon)-1/N$ with $\varepsilon\in(0,1)$ and $N\varepsilon$ being a rational number but not an integer in the interval polyhedral ambiguity set $\P_{I}$. According to Corollary~\ref{mixed_representable_upper_lower_fcvar}, the proposed DFO \eqref{dfo} is still MICP-R. Alternatively, the two-stage program \eqref{second_stage_trp_obj} with the interval polyhedral ambiguity set $\P_{I}$ admits a naive bilinear formulation, which can be solved directly by Gurobi. Since we cannot solve the bilinear model to optimality within the time limit, we use GAP to denote its optimality gap as
$
\textrm{GAP} (\%)= (| \textrm{UB} -\textrm{LB} | )/|\textrm{LB}|\times 100,
$ 
where $``\textrm{UB}" $ and $``\textrm{LB}" $ denotes the best upper bound and the best lower bound found by Gurobi.  We repeat the solution process $5$ times and display the average performance result in Table~\ref{tab_comparision_micpr}. We find that the MICP-R formulation can improve the running time significantly, even for small-scale instances, which shows the effectiveness of exploring the MICP-R formulation.

\begin{table}[htbp]
\centering
\caption{Comparisons Between DFO \eqref{dfo} in Corollary~\ref{mixed_representable_upper_lower_fcvar} and Its Bilinear Counterpart.}
\setlength{\tabcolsep}{1pt} 
\renewcommand{\arraystretch}{1} 
\label{tab_comparision_micpr}
\small
\begin{center}
\begin{tabular}{|c|r|r|r|r|r|r|}
\hline
\multirow{3}{*}{$N$} & \multicolumn{3}{c|}{$\varepsilon=0.16$}  & \multicolumn{3}{c|}{$\varepsilon=0.18$} \\ \cline{2-7}
& \multicolumn{2}{c|}{Bilinear} & \multirow{2}{*}{\makecell{DFO \eqref{dfo} in Corollary~\ref{mixed_representable_upper_lower_fcvar} \\Time (s)}} &  \multicolumn{2}{c|}{Bilinear} & \multirow{2}{*}{\makecell{DFO \eqref{dfo} in Corollary~\ref{mixed_representable_upper_lower_fcvar} \\Time (s)} } \\ \cline{2-7}
&  Time (s) &   GAP(\%) &   &  Time (s) &   GAP(\%) &  \\ \hline
45  &  3600 &   3.04& 32.73 &  3600 &  4.39  & 46.94 \\ \cline{1-7} 
55  &3600   &   5.35& 45.92 &  3600 &  6.24  &  57.93\\ \cline{1-7} 
65&   3600&   7.49& 58.38 &  3600 &  8.93 & 72.96  \\ \hline
\end{tabular}
\end{center}
\vspace{-2em}
\end{table}

\noindent\textbf{Experiment 2. Value of MICP-R using Hurwicz Model.} Based on Corollary~\ref{mixed_representable_upper_lower_fcvar}, we can also provide MICP-R formulations under the Hurwicz criterion for the TRA problem \eqref{eq_tcfl}. That is, we consider the following problem:
\begin{align}
\min _{\bm x\geq \bm 0} \left\{\bar{\lambda}\inf_{\Pr\in\P_{I}}\E_{\Pr}\left[Q(\bm x, \trxi)\right]+(1-\bar{\lambda})\sup_{\Pr\in\P_{I}}\E_{\Pr}\left[Q(\bm x, \trxi)\right]\right\},\label{second_stage_trp_obj_Hurwicz}
\end{align}
where $\bar{\lambda}\in[0,1]$ is a known level of optimism. Notice that when $\bar{\lambda}=0$ in Hurwicz \eqref{second_stage_trp_obj_Hurwicz}, we obtain the DRO formulation for TRA problem \eqref{eq_tcfl} as below
\begin{align}
\min _{\bm x\geq \bm 0} \sup_{\Pr\in\P_{I}}\E_{\Pr}\left[Q(\bm x, \trxi)\right].\label{second_stage_trp_obj_dro}
\end{align}
In this experiment, we compare the solutions from DFO \eqref{second_stage_trp_obj_dfo}, DRO \eqref{second_stage_trp_obj_dro}, and Hurwicz \eqref{second_stage_trp_obj_Hurwicz} via out-of-sample performances. Specifically, after solving the corresponding DFO \eqref{second_stage_trp_obj_dfo}, DRO \eqref{second_stage_trp_obj_dro}, and Hurwicz \eqref{second_stage_trp_obj_Hurwicz}, we generate additional $100$ random testing instances to evaluate the solution performances, i.e., to assess the performance of the first-stage decision of each model. Training and test instances are generated in the same manner, i.e., the components of the customer demand $\tilde{\bm{\xi}}$ are i.i.d.~truncated Gaussian with mean $\bar{d}/n_1$ and variance $0.005\times\bar{d}$, $\bar{d}=1000$. We record all the $50\%, 60\%, 70\%, 80\%, 90\%$ quantiles of the second-stage values, respectively. We then report each quantile's $95\%$ asymptotic confidence interval (C.I.) among these $100$ testing instances. We set $\bar{\lambda}=\{0.2,0.4,0.6,0.8\}$ in Hurwicz \eqref{second_stage_trp_obj_Hurwicz} and consider the training sample size $N=45$ with $\epsilon=0.16$. The results are shown in Table~\ref{table_model_comparison_hurwicz_stationary} and Figure~\ref{Fig_Hurwicz_stationary_figure}. 
In this case, DFO \eqref{second_stage_trp_obj_dfo} is consistently better than other methods when comparing $50\%, 60\%, 70\%$, and $80\%$ quantiles of the second-stage values. Hurwicz \eqref{second_stage_trp_obj_Hurwicz} with $\bar{\lambda}=0.8$ performs better than other methods when comparing $90\%$ quantile of the second-stage values. By carefully choosing an optimism level $\bar{\lambda}$, it is seen that the Hurwicz \eqref{second_stage_trp_obj_Hurwicz} can reduce the conservatism (in this case, we can choose $\bar{\lambda}=0.8$). 

\begin{table}[htbp]
\vspace{-1.1em}
\centering
\caption{Quantile Comparisons among DFO, DRO, and Hurwicz Models in Experiment 2.}
\renewcommand{\arraystretch}{1} 
\label{table_model_comparison_hurwicz_stationary}
\scriptsize
\begin{center}
\begin{tabular}{|c|rrrrr|}
\hline
\multirow{2}{*}{Model} & \multicolumn{5}{c|}{Quantile}                             \\ \cline{2-6} 
& \multicolumn{1}{c|}{50\%}                   & \multicolumn{1}{c|}{60\%}                   & \multicolumn{1}{c|}{70\%}                   & \multicolumn{1}{c|}{80\%}                   & \multicolumn{1}{c|}{90\%} \\ \hline
DRO     \eqref{second_stage_trp_obj_dro}     & \multicolumn{1}{r|}{{[}1826.63, 1826.95{]}} & \multicolumn{1}{r|}{{[}1828.39, 1828.73{]}} & \multicolumn{1}{r|}{{[}1830.33, 1830.66{]}} & \multicolumn{1}{r|}{{[}1832.75, 1833.14{]}} & {[}1835.82,1836.23{]}     \\ \hline
Hurwicz \eqref{second_stage_trp_obj_Hurwicz} with $\bar{\lambda}=0.2$ & \multicolumn{1}{r|}{{[}1771.41, 1771.74{]}} & \multicolumn{1}{r|}{{[}1773.23, 1773.58{]}} & \multicolumn{1}{r|}{{[}1775.22, 1775.56{]}} & \multicolumn{1}{r|}{{[}1777.67, 1778.07{]}} & {[}1780.75,1781.16{]}     \\ \hline
Hurwicz \eqref{second_stage_trp_obj_Hurwicz} with $\bar{\lambda}=0.4$ & \multicolumn{1}{r|}{{[}1771.21, 1771.55{]}} & \multicolumn{1}{r|}{{[}1773.07, 1773.42{]}} & \multicolumn{1}{r|}{{[}1775.07, 1775.40{]}} & \multicolumn{1}{r|}{{[}1777.53, 1777.93{]}} & {[}1780.63,1781.03{]}     \\ \hline
Hurwicz \eqref{second_stage_trp_obj_Hurwicz} with $\bar{\lambda}=0.6$ & \multicolumn{1}{r|}{{[}1766.25, 1766.60{]}} & \multicolumn{1}{r|}{{[}1768.11, 1768.47{]}} & \multicolumn{1}{r|}{{[}1770.13, 1770.47{]}} & \multicolumn{1}{r|}{{[}1772.61, 1773.00{]}} & {[}1775.70,1776.11{]}     \\ \hline
Hurwicz \eqref{second_stage_trp_obj_Hurwicz} with $\bar{\lambda}=0.8$ & \multicolumn{1}{r|}{{[}1763.06, 1763.41{]}} & \multicolumn{1}{r|}{{[}1764.94, 1765.31{]}} & \multicolumn{1}{r|}{{[}1767.02, 1767.37{]}} & \multicolumn{1}{r|}{{[}1769.49, 1769.90{]}} & {[}1772.86,1773.43{]}     \\ \hline
DFO   \eqref{second_stage_trp_obj_dfo}                     & \multicolumn{1}{r|}{{[}1757.02, 1757.37{]}} & \multicolumn{1}{r|}{{[}1758.95, 1759.34{]}} & \multicolumn{1}{r|}{{[}1761.22, 1761.67{]}} & \multicolumn{1}{r|}{{[}1768.88, 1769.01{]}} & {[}1774.94,1776.10{]}     \\ \hline
\end{tabular}
\end{center}
\end{table}

\begin{figure}[htbp]
\vspace{-1.1em}
\begin{center}
\centering
\includegraphics[width=0.8\textwidth]{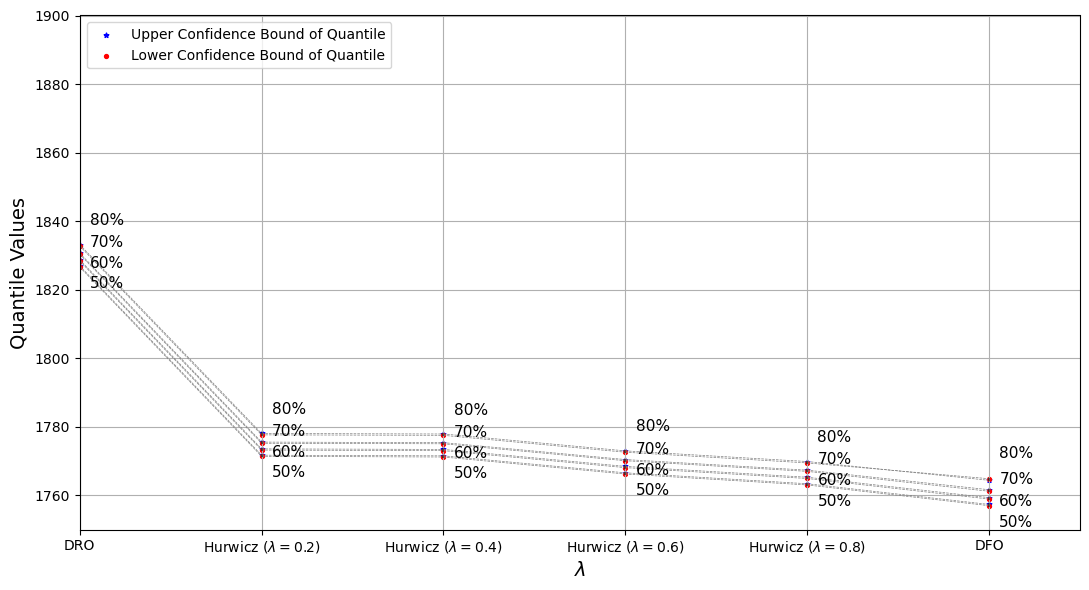}
\vspace{-1.1em}
\caption{Illustration of Quantile Comparisons in  Experiment 2}
\label{Fig_Hurwicz_stationary_figure}
\end{center}
\vspace{-2em}
\end{figure}


\noindent{\textbf{Experiment 3. Model Comparisons When the Testing Distribution is Different From the Training One.}}  We follow the same procedure described in Experiment 2, i.e., we record all the $50\%$, $60\%$, $70\%$, $80\%, 90\%$ quantiles in the second-stage scenarios for each method (e.g., DFO, DRO, and Hurwicz models) in each testing instance, respectively, and report the average of each quantile among these $100$ randomly generated testing instances. The testing and training setting are the same as that of Experiment 2, except that we assume that the components of the customer demand $\tilde{\bm{\xi}}$ are i.i.d.~truncated Gaussian with mean $1000/n_1$ and variances $0.01\times 1000$. We repeat the procedure in Experiment 2 to solve Hurwicz \eqref{second_stage_trp_obj_Hurwicz} with $\bar{\lambda}=\{0,0.2,0.4,0.6,0.8\}$ to
better capture the influence of the optimism level. The results are shown in Table~\ref{table_model_comparison_hurwicz_nonstationary} and Figure~\ref{Fig_Hurwicz_nonstationary_figure}.
As anticipated, Hurwicz \eqref{second_stage_trp_obj_Hurwicz}  can alleviate conservatism and improve out-of-sample performance when the testing distribution has perturbations by selecting the level of optimism $\bar{\lambda}$ (in this case, we can choose $\bar{\lambda}=0.2$ or $0.4$).

\begin{table}[htbp]
\vspace{-1.1em}
\centering
\caption{Quantile Comparisons among DFO, DRO, and Hurwicz Models in Experiment 3.}
\renewcommand{\arraystretch}{1} 
\label{table_model_comparison_hurwicz_nonstationary}
 \scriptsize
\begin{center}
\begin{tabular}{|c|rrrrr|}
\hline
\multirow{2}{*}{Model} & \multicolumn{5}{c|}{Quantile}                                              \\ \cline{2-6} 
& \multicolumn{1}{c|}{50\%}                    & \multicolumn{1}{c|}{60\%}                   & \multicolumn{1}{c|}{70\%}                   & \multicolumn{1}{c|}{80\%}                   & \multicolumn{1}{c|}{90\%} \\ \hline
DRO  \eqref{second_stage_trp_obj_dro}                  & \multicolumn{1}{r|}{{[}1827.91,   1828.75{]}} & \multicolumn{1}{r|}{{[}1831.56, 1832.45{]}} & \multicolumn{1}{r|}{{[}1835.51, 1836.45{]}}  & \multicolumn{1}{r|}{{[}1840.37, 1841.31{]}} & {[}1846.95,1847.93{]}     \\ \hline
Hurwicz \eqref{second_stage_trp_obj_Hurwicz} with $\bar{\lambda}=0.2$ & \multicolumn{1}{r|}{{[}1772.65, 1773.51{]}}  & \multicolumn{1}{r|}{{[}1776.35, 1777.27{]}} & \multicolumn{1}{r|}{{[}1780.36, 1781.34{]}} & \multicolumn{1}{r|}{{[}1785.69, 1786.85{]}} & {[}1815.39,1833.73{]}     \\ \hline
Hurwicz \eqref{second_stage_trp_obj_Hurwicz} with $\bar{\lambda}=0.4$ & \multicolumn{1}{r|}{{[}1772.59, 1773.47{]}}  & \multicolumn{1}{r|}{{[}1776.38, 1777.29{]}} & \multicolumn{1}{r|}{{[}1780.49, 1781.47{]}} & \multicolumn{1}{r|}{{[}1785.76, 1786.92{]}} & {[}1815.58,1833.99{]}     \\ \hline
Hurwicz \eqref{second_stage_trp_obj_Hurwicz} with $\bar{\lambda}=0.6$ & \multicolumn{1}{r|}{{[}1767.90, 1768.81{]}}    & \multicolumn{1}{r|}{{[}1771.77, 1772.71{]}} & \multicolumn{1}{r|}{{[}1776.25, 1777.55{]}} & \multicolumn{1}{r|}{{[}1785.61, 1792.29{]}} & {[}1900.45,1927.38{]}     \\ \hline
Hurwicz \eqref{second_stage_trp_obj_Hurwicz} with $\bar{\lambda}=0.8$ & \multicolumn{1}{r|}{{[}1764.88, 1765.81{]}}  & \multicolumn{1}{r|}{{[}1768.92, 1769.95{]}} & \multicolumn{1}{r|}{{[}1773.96, 1778.12{]}} & \multicolumn{1}{r|}{{[}1812.27, 1828.57{]}} & {[}1972.28,2000.03{]}     \\ \hline
DFO   \eqref{second_stage_trp_obj_dfo}                    & \multicolumn{1}{r|}{{[}1759.27, 1760.42{]}}  & \multicolumn{1}{r|}{{[}1765.92, 1770.52{]}} & \multicolumn{1}{r|}{{[}1804.41, 1824.36{]}} & \multicolumn{1}{r|}{{[}1931.37, 1956.21{]}} & {[}2111.82,2139.61{]}     \\ \hline
\end{tabular}
\end{center}
\end{table}

\begin{figure}[htbp]
\begin{center}
\centering
\includegraphics[width=0.8\textwidth]{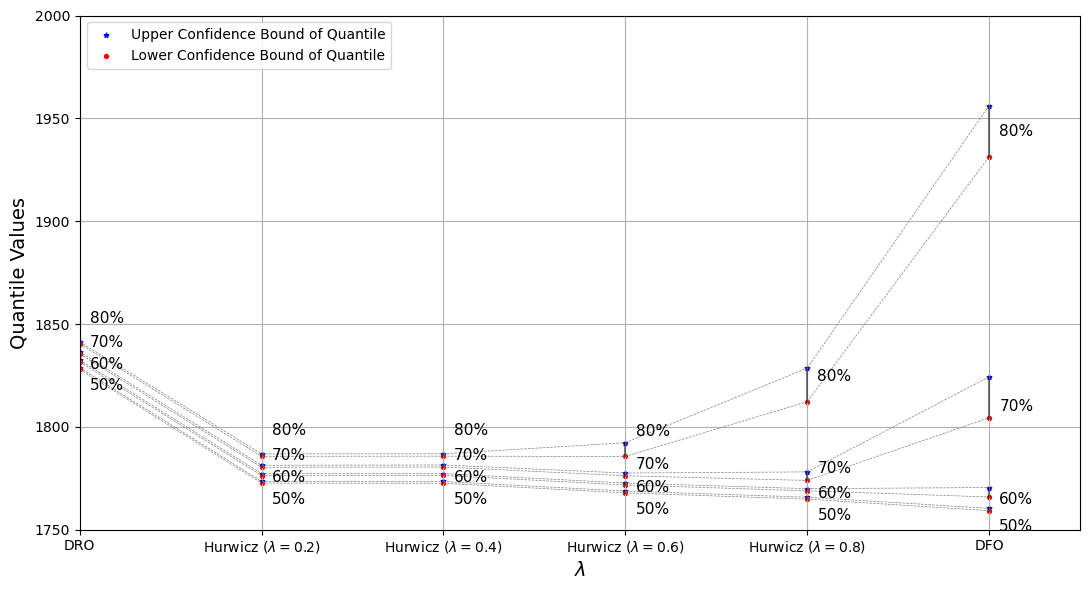}
\vspace{-1.1em}
\caption{Illustration of Quantile Comparisons in Experiment 3}
\label{Fig_Hurwicz_nonstationary_figure}
\end{center}
\vspace{-2em}
\end{figure}

\section{Conclusion} \label{sec_conclusion}
This paper provided sufficient and necessary conditions where DFO can be tractable or intractable. Even though DFO is NP-hard to solve in general, we demonstrated that many DFO problems can be mixed-integer convex programming representable, which can be solved by off-the-shelf solvers. We numerically demonstrated the effectiveness of using MICP-R formulations. One future direction is to extend the results to the two-stage stochastic programs under decision-dependent uncertainty. It is also interesting to investigate the theoretical advantages of the Hurwicz criterion.

\section*{Acknowledgments}
This research has been supported in part by the National Science Foundation grants 2246414 and 2246417 and the Office of Naval Research grant N00014-24-1-2066.

\bibliographystyle{plain}
\bibliography{dfo.bib}

\end{document}